\documentclass[12pt,reqno]{amsart}

\newcommand\version{June 16, 2026}
\usepackage[utf8]{inputenc}
\usepackage[british,UKenglish]{babel}

\usepackage{amsmath,amsfonts,amsthm,amssymb,amsxtra,mathtools}
\usepackage{hyperref}
\usepackage{color}
\usepackage{bbm} 
\usepackage{bm} 
\usepackage[normalem]{ulem}
\usepackage[font=small,labelfont=bf]{caption}
\usepackage[autostyle=false, style=english]{csquotes}
\MakeOuterQuote{"}

\mathtoolsset{showonlyrefs} 

\newtheorem{theorem}{Theorem}[section]
\newtheorem{proposition}[theorem]{Proposition}
\newtheorem{lemma}[theorem]{Lemma}
\newtheorem{corollary}[theorem]{Corollary}

\theoremstyle{definition}

\newtheorem{definition}[theorem]{Definition}

\theoremstyle{remark}

\newtheorem{remark}[theorem]{Remark}

\numberwithin{equation}{section}

\newcommand{\C}{\mathbb{C}}

\renewcommand{\epsilon}{\varepsilon}

\newcommand{\loc}{{\rm loc}}

\newcommand{\N}{\mathbb{N}}

\renewcommand{\phi}{\varphi}
\newcommand{\R}{\mathbb{R}}

\newcommand{\Z}{\mathbb{Z}}

\newcommand{\ceil}[1]{\left \lceil{#1}\right \rceil }

\DeclareMathOperator{\diam}{diam}
\DeclareMathOperator{\dist}{dist}

\DeclareMathOperator{\im}{Im}

\DeclareMathOperator{\re}{Re}
\DeclareMathOperator{\spec}{spec}
\DeclareMathOperator{\supp}{supp}

\def\rd{\mathrm{d}}

\newcommand{\one}{\mathbf{1}}

\newcommand{\e}{{\rm e}}
\newcommand\I{\mathrm{i}}

\newcommand{\Ran}{\mathop{\rm Ran}}

\begin{document}

\title[Random Schr\"odinger operators on manifolds --- \version]{Random Schr\"odinger operators on manifolds and abstract bounds for multiplier-type operators}

\author[J.-C. Cuenin]{Jean-Claude Cuenin}
\address[Jean-Claude Cuenin]{Department of Mathematical Sciences, Loughborough University, Loughborough,
 Leicestershire, LE11 3TU United Kingdom}
\email{J.Cuenin@lboro.ac.uk}

\author[K. Merz]{Konstantin Merz}
\address[Konstantin Merz]{Institute for Theoretical Physics, ETH Zurich, Wolfgang-Pauli-Strasse 27, 8093 Zurich, Switzerland, and Institut f\"ur Analysis und Algebra, Technische Universit\"at Braunschweig, Universit\"atsplatz 2, 38106 Braun\-schweig, Germany}
\email{k.merz@tu-bs.de, k.merz@phys.ethz.ch}

\author[E. Stefanescu]{Eduard Stefanescu}
\address[Eduard Stefanescu]{Institut für Analysis und Zahlentheorie, TU Graz, Steyrergasse 30, 8010 Graz, Austria}
\email{eduard.stefanescu@tugraz.at}

\thanks{\copyright\, 2026 by the authors. This paper may be reproduced, in its entirety, for non-commercial purposes. J.-C. C. acknowledges support through the Engineering \& Physical Sciences Research Council (EP/X011488/1). E.~S.~acknowledges support through the Austrian Science Fund (FWF) [Grant-DOI 10.55776/P35322, 10.55776/PAT5120424, and 10.55776/PAT4632823].}

\subjclass[2020]{Primary: 35P20; Secondary: 35J10, 47B80, 58J50}
\keywords{Random Schrödinger operators, spectral inclusion bounds, resolvent estimates, spectral clusters, square root cancellation}

\date{\version}

\begin{abstract}
We study random Schrödinger operators on closed Riemannian manifolds with Anderson-type potentials. We prove high-probability spectral inclusion bounds showing that eigenvalues remain close to those of the Laplacian, with deviations controlled by a norm of the potential coefficients. Compared with deterministic bounds, this yields a square-root cancellation gain. The proof is based on a general principle showing that randomisation improves operator norm bounds for multiplier-type operators, which we formulate in both discrete and continuous settings.
\end{abstract}

\maketitle

\section{Introduction and main results}
\subsection{Random Schrödinger operators}
On a closed (i.e., compact, boundaryless) Riemannian manifold $(M,g)$ of dimension $d \geq 2$, we consider the random Schr\"odinger operator
\begin{align}
  H_{\omega,\lambda} = -\Delta_g + V_{\omega,\lambda}
  \quad \text{on } L^2(M),
\end{align}
where $V_{\omega,\lambda}$ is an Anderson-type potential built from bumps at scale $\lambda^{-1}$.
More precisely, we assume that
\begin{align}\label{eq:Anderson-type}
  V_{\omega,\lambda}(x)
  = \sum_{j=1}^{N(\lambda)} \omega_j v_j(\lambda) \psi_j(x;\lambda),
  \quad x \in M,
\end{align}
where the functions $\psi_j(\cdot;\lambda)$ are uniformly bounded and have bounded overlap, with spatial support at scale \(\lambda^{-1}\), $v_j(\lambda) \in \mathbb{C}$ are deterministic coefficients, the random variables $\omega_j$ are i.i.d.~symmetric Bernoulli (Rademacher) or centred normalised Gaussians, $\lambda\geq 2$ is a large parameter (the inverse randomisation length), and $N(\lambda)\leq C\lambda^d$. 

Since $M$ is compact, the spectrum of $-\Delta_g$ is discrete 
\[
\spec(-\Delta_g)=\{\lambda_k^2\}_{k=0}^\infty,
\qquad
0=\lambda_0^2\leq \lambda_1^2\leq \cdots,
\qquad
\lambda_k^2\to \infty.
\]
Our first main result gives a high-probability bound on the location of the eigenvalues of $H_{\omega,\lambda}$, up to energy $\lambda^2$, 
in terms of the \(\ell^{2q}\)-norm of \(v(\lambda)=(v_j(\lambda))_{j=1}^{N(\lambda)}\).

\begin{theorem}\label{thm:main_intro}
Fix \(N\geq 1\). Then there exists a constant \(C_N\) such that for every $d/2\leq q\leq\infty$, every $\lambda\geq 2$, and every \(K\geq 1\), with probability at least \(1-\exp(-K^2)\), 
we have
\[
\{z\in\spec(H_{\omega,\lambda}) : |z|\leq \lambda^2,\,\dist(z, \spec(-\Delta_g))\geq \lambda^{-N}\}
\subset
\bigcup_{\lambda_k\leq \lambda} D(\lambda_k^2,C_NKr_k)\cup \Omega,
\]
where 
\begin{align}\label{def.r_k}
r_k
:=
(1+\lambda_k)^{\frac{d(q+1)}{2q}-\mu(q)}  \, \lambda^{-\frac{d(q+1)}{2q}} (\log\lambda)^{11/2}\|v(\lambda)\|_{\ell^{2q}},
\end{align}
\begin{align}
\Omega:=\{z\in\C: |z|^{1/2}(1+|z|)^{-\frac{d(q+1)}{4q}+\frac{\mu(q)}{2}}  \leq C_NK\, \lambda^{-\frac{d(q+1)}{2q}} (\log \lambda)^{11/2} \|v(\lambda)\|_{\ell^{2q}}\},
\end{align}
and  
\begin{align}\label{def.mu(q)}
      \mu(q):=\begin{cases}
1, & 1\leq q\leq \frac{d+1}{2},\\
\frac{d+1}{4q}+\frac{1}{2}, & \frac{d+1}{2}\leq q\leq\infty.
\end{cases}
  \end{align}
\end{theorem}





We postpone further remarks concerning Theorem~\ref{thm:main_intro} to the end of the introduction, including a comparison with the deterministic results of \cite{Cuenin2026}.


\subsection{Multiplier-type operators}
The main new ingredient in the proof of Theorem~\ref{thm:main_intro} is an abstract operator-theoretic principle showing that randomisation improves norm bounds for multiplier-type operators of the form \(T_1^* V T_2\), where \(V\) is a multiplication operator and \(T_1,T_2\) are structured linear operators, for instance spectral projectors, resolvents, or Fourier extension operators.

Random multiplier estimates play a crucial role in the work of Schlag--Shubin--Wolff \cite{Schlagetal2002} and Bourgain \cite{Bourgain2002,Bourgain2003} on weakly disordered random Schr\"odinger operators on \(\mathbb{Z}^2\). In Schlag--Shubin--Wolff, the relevant bounds are obtained by exploiting geometric considerations involving curvature. Bourgain introduced a different approach, based on a general entropy bound, which works in arbitrary dimension and does not rely on curvature. On the other hand, Bourgain's method makes essential use of the Rademacher or Gaussian distribution of the random variables \(\omega_j\).

Our contribution is a general abstract framework inspired by Bourgain's approach. An instance of such a principle for operators on \(\mathbb{R}^d\) was established in \cite{CueninMerz2025}. In the present paper, we develop this idea in a substantially broader form.

We formulate the abstract results in two settings. The first is a discrete model, where a function on a finite set is randomised over the elements of a partition. The second is a continuous analogue on Ahlfors regular metric measure spaces, where the randomisation is performed on Voronoi cells associated with a finite set. These results are of independent interest and form the abstract core of the paper.

\begin{theorem}[Discrete abstract estimate]\label{thm:abstract1_intro}
For $i=1,2$, let $\mathcal{H}_i$ be complex separable Hilbert spaces, $E$ a nonempty finite set, and consider bounded linear operators
\[
S_i:\mathcal H_i\to \ell^\infty(E),
\]
which we also view as operators into $\ell^{p'}(E)$ since $E$ is finite.
Let $\{E_j\}_{j\in J}$ be a partition of $E$, that is,
\[
E=\bigcup_{j\in J}E_j,
\qquad
E_j\cap E_k=\emptyset \quad \text{for } j\neq k.
\]
Let $v\in \ell^\infty(E)$, set $v_j:=v\mathbf{1}_{E_j}$, and let $\{\omega_j\}_{j\in J}$ be i.i.d.~symmetric Bernoulli or centred normalised Gaussian random variables. Define
\begin{align}\label{def:v_omega discrte}
v_{\omega}(x)=\omega_j v(x),
\qquad x\in E_j.
\end{align}
Let $2\leq q\leq\infty$ and $1\leq p\leq 2\leq p'\leq \infty$ satisfy $\frac{1}{q}=\frac{1}{p}-\frac{1}{p'}$. Then
\begin{align}\label{eq:abstract1}
\!\!
\mathbf{E}\|S_1^*v_{\omega}S_2\|_{\mathcal{H}_2\to \mathcal{H}_1}
\lesssim
(\log \#E)^{\frac{5}{2}}
\|S_1\|_{\mathcal{H}_1\to \ell^{p'}(E)}
\|S_2\|_{\mathcal{H}_2\to \ell^\infty(E)}
\Big(\sum_{j\in J}\|v_j\|_{\ell^p(E_j)}^{2q}\Big)^{\frac{1}{2q}},
\end{align}
with the obvious modification for $q=\infty$.
\end{theorem}

Here and below, we use the notation $A \lesssim B$ for functions $A,B\geq0$ to indicate that there is a constant $c$ such that $A\leq c B$. If $c$ depends on a parameter $\tau$, we may write $A\lesssim_\tau B$, but the dependence on $d,M$ is usually ignored in the notation.

\begin{remark}\label{rem:Holder}
By H\"older's inequality,
\[
\Big(\sum_{j\in J}\|v_j\|_{\ell^p(E_j)}^{2q}\Big)^{1/2q}
\leq
\sup_{j\in J}(\# E_j)^{\frac{1}{p}-\frac{1}{2q}}\|v\|_{\ell^{2q}(E)}.
\]
In particular, if all the $E_j$ are singletons, then the right-hand side is just $\|v(\lambda)\|_{\ell^{2q}(E)}$.
\end{remark}

We now state a continuous version of Theorem \ref{thm:abstract1_intro}. All technical notions appearing in the statement are defined in Section~\ref{sec:abstract}. A precise and stronger statement is given in Theorem \ref{thm:abstract2}.

\begin{theorem}[Continuous abstract estimate]\label{thm:abstract2_intro}
Let $X$ be an Ahlfors $d$-regular metric measure space, and let
\[
T_i:\mathcal{H}_i\to L^\infty(X)\cap L^{p'}(X), \qquad i=1,2,
\]
be bounded linear operators, where $\mathcal{H}_i$ are complex separable Hilbert spaces. Let $\Lambda$ be a finite, maximal $r$-separated subset of $X$, for some $r>0$, and let $V_\omega$ be the randomisation, by i.i.d.~symmetric  Bernoulli or centred normalised Gaussian random variables, of a measurable complex-valued function $V\in L^1_{\mathrm{loc}}(X)$ on the Voronoi cells associated with $\Lambda$. Assume moreover that $X$, $\Lambda$, and the ranges of $T_1$ and $T_2$ satisfy the uniform reparametrisation and local constancy hypotheses stated in Section~\ref{sec:abstract}. Then, whenever $2\leq q\leq\infty$ and $1\leq p\leq 2\leq p'\leq \infty$ satisfy $\frac{1}{q}=\frac{1}{p}-\frac{1}{p'}$, one has
\begin{align*}
\mathbf{E}\|T_1^*V_{\omega}T_2\|_{\mathcal{H}_2\to \mathcal{H}_1}
&\lesssim
r^{d/p}
(\log \#\Lambda)^{5/2}
\|T_1\|_{\mathcal{H}_1\to L^{p'}(X)}
\|T_2\|_{\mathcal{H}_2\to L^\infty(X)}
\\
&\qquad \times
\Big(\sum_{x\in \Lambda}\|V\|_{L^\infty(B(x,2r))}^{2q}\Big)^{1/2q},
\end{align*}
with the obvious modification for $q=\infty$.
\end{theorem}
In the compact-manifold setting, there are two principal choices of \(T_1,T_2\) for which this theorem becomes effective.

\medskip

\noindent
\textbf{(i) Spectral projectors.}
Let
\[
P=\sqrt{-\Delta_g},
\qquad
\Pi_\lambda=\mathbf 1(P\in[\lambda,\lambda+1]).
\]
The range of \(\Pi_\lambda\) is locally constant at wavelength scale \(r=\lambda^{-1}\), and Sogge's spectral cluster bounds give
\[
\|\Pi_\lambda\|_{L^2(M)\to L^{\frac{2(d+1)}{d-1}}(M)}
\lesssim
\lambda^{\frac{d-1}{2(d+1)}}.
\]
Inserting these estimates into Theorem~\ref{thm:abstract2_intro} yields the  randomised spectral cluster estimate
    \begin{align*}
    \mathbf{E}\|\Pi_{\lambda} ^*V_{\omega,\lambda}\Pi_{\lambda} \|_{L^2(M)\to L^2(M)}
    & \lesssim \lambda^{-1}(\log \lambda )^{5/2}\|v(\lambda)\|_{\ell^{d+1}},
  \end{align*}  
see Proposition \ref{prop: manifold spectral cluster expectation}. 
This is the Riemannian analogue of the random Fourier extension estimate \cite[Section 6]{CueninMerz2025}. By sub-Gaussian tail bounds (see Lemma \ref{lem:tailbound}), it follows that with probability at least $1-\exp(-K^2)$, 
\[
\left|\int_M |\Pi_{\lambda}u|^2V_{\omega,\lambda}\rd v_g\right|\lesssim K \lambda^{-1}(\log \lambda )^{5/2}\|v(\lambda)\|_{\ell^{d+1}}\|u\|_{L^2(M)}^2,
\]
while the deterministic bound holds only with the ${\ell^{\frac{d+1}{2}}}$ norm, without the logarithm. 

Whereas \(L^2\to L^{p'}\) bounds measure the size of level sets of \(\Pi_\lambda u\), weighted expressions of the form
\(\int_M |\Pi_\lambda u|^2 V\,\mathrm{d}v_g\)
 provide geometric information about the shape of level sets. In the present random setting, however, \(V_{\omega,\lambda}\) is signed, so this geometric interpretation is less direct.

In the Euclidean setting, weighted Fourier restriction theory is concerned with inequalities of the form
\[
\int_{\R^d}|E_Sf|^2w\,\rd x\leq C(w)\|f\|_{L^2(S)}^2
\]
for deterministic nonnegative weights \(w\). We refer, for example, to \cite{MR4980305,MR4759602} for further background and results.

\medskip

\noindent
\textbf{(ii) Half-resolvents.}
We will prove deterministic estimates for the square root of the resolvent,
\[
\bigl\||-\Delta_g-z|^{-1/2}\bigr\|_{L^2(M)\to L^{\frac{2(d+1)}{d-1}}(M)}
\lesssim
\max(d(z)^{-1/2},\langle z\rangle^{-1/4})
    \log\langle z\rangle^{1/2}
    \langle z\rangle^{\frac{d-1}{4(d+1)}},
\]
see Proposition \ref{prop:L2_to_Lp'_resolvent_estimate}.
Combining these bounds with Theorem~\ref{thm:abstract2_intro} yields a probabilistic estimate for the Birman-Schwinger operator
\[
\mathcal K_{\omega,\lambda}(z)
=
|-\Delta_g-z|^{-1/2}
V_{\omega,\lambda}
(-\Delta_g-z)^{-1/2},
\]
see Proposition \ref{prop:BS_small_q}. The spectral inclusion in Theorem \ref{thm:abstract1_intro} is then obtained by combining this estimate with the Birman--Schwinger principle.

\subsection{Further applications}

The compact-manifold Schr\"odinger problem considered in this paper is only one instance of a more general mechanism that separates the probabilistic from the analytic input.
As an outlook, we briefly indicate a few problems in which we expect this mechanism to apply.

\begin{itemize}
    \item[(i)] \textbf{Random Schr\"odinger operators in other geometric settings.}
    The argument is not specific to compact boundaryless manifolds. It should extend to compact manifolds with boundary, as well as to noncompact spaces such as asymptotically conic or hyperbolic manifolds. The appropriate deterministic spectral cluster and resolvent bounds may be found, e.g., in \cite{MR1257279,MR2316270,MR3805767,MR4150258}.
    Another direction would be to consider random operators on $\R^d$ with trapping background potentials or magnetic fields, see, e.g., \cite{MR2140267,MR2314091}. 
    
    In all of these settings, the length of the spectral cluster window and the proximity of $z$ to the spectrum of the unperturbed operator in the half-resolvent estimates, both of which play an important role in the proof of Theorem \ref{thm:abstract1_intro}, are expected to shrink compared to the compact manifold setting. One then also expects the $r_k$ in Theorem \ref{thm:abstract1_intro} to shrink. This improvement could also be explored for the torus, where shrinking spectral cluster bounds are well-studied, see, e.g., \cite{MR4422439,MR4818473}. 

    Once the location of individual eigenvalues is well understood, one could then investigate their distribution via Lieb--Thirring type inequalities; this was done in the Euclidean setting in \cite{CueninMerz2023}.  

    \item[(ii)] \textbf{Existence and completeness of a.s.\ wave operators in $\R^d$.} Ionescu and Schlag \cite{IonescuSchlag2006} proved the existence and completeness of wave operators for Schr\"odinger operators on $\R^d$ for a large class of deterministic potentials. In particular, their results apply when $V\in L^{\frac{d+1}{2}}(\R^d)$. When the potential is random, we expect an improvement to $L^{d+1-\epsilon}(\R^d)$. This would be an analogue of Bourgain's result \cite{Bourgain2003} for $\Z^2$;
    the first author and Schippa \cite{CueninSchippa2022} also sketched the argument for $\Z^3$. A related question is whether the results of Ionescu--Jerison \cite{IonescuJerison2003} and Koch--Tataru \cite{KochTataru2006} on the absence of embedded eigenvalues for Schr\"odinger operators on $\R^d$ can be strengthened under randomisation. Finally, it would be interesting to investigate the \(L^p\) boundedness of wave operators for random potentials.

    \item[(iii)] \textbf{Discrete models.}
    The discrete abstract estimate, Theorem~\ref{thm:abstract1_intro}, is designed for applications to lattice and graph operators. In particular, it is compatible with spectral projectors and resolvents for discrete Schr\"odinger operators, and may therefore be useful in weak-disorder problems on \(\mathbb Z^d\) and on more general finite-range graph models. This direction is close in spirit to the work of Schlag--Shubin--Wolff and Bourgain discussed above.
\end{itemize}

\subsection{Comparison of Theorem~\ref{thm:main_intro} with deterministic results}\label{comparison between the random and deterministic results}

In \cite{Cuenin2026}, the first author proved 
\begin{align}\label{spectral inclusion}
  \spec(-\Delta_{g}+V)\subset \bigcup_{k=0}^{\infty}D(\lambda_k^2,C\tilde{r}_k^{\rm det})\cup \{z\in\C: |z|^{\frac{1}{2}}(1+|z|)^{-\sigma(q)}\leq C \|V\|_{L^q(M)}\},
\end{align}
where $d/2<q< \infty$,
\(
\tilde{r}_k^{\rm det}:=\|V\|_{L^q(M)}(1+\lambda_k)^{2\sigma(q)},
\)
and $\sigma(q)$ is given by \eqref{sigma(q)}. Specializing to Anderson-type potentials \eqref{eq:Anderson-type} with $\omega_j=1$ and applying this with $q$ replaced by $2q$, it follows that
\[
\{z\in\spec(H_{\omega,\lambda}) : |z|\leq \lambda^2\}
\subset
\bigcup_{\lambda_k\leq \lambda} D(\lambda_k^2,CKR_k')\cup \Omega^{\rm det},
\]
where $d/2<2q< \infty$ and
\[
r_k^{\rm det}:=(1+\lambda_k)^{2\sigma(2q)}\lambda^{-\frac{d}{2q}}\|v(\lambda)\|_{\ell^{2q}},
\]
\[
\Omega^{\rm det}:=\{z\in\C: |z|^{\frac{1}{2}}(1+|z|)^{-\sigma(2q)}\leq C\lambda^{-\frac{d}{2q}}\|v(\lambda)\|_{\ell^{2q}}\}.
\]

 For $\lambda_k\ll \lambda$, one has $\frac{r_k}{r_k^{\rm det}}\ll 1$; more precisely,
\begin{align}\label{eq: r_k/r_kdet}
\frac{r_k}{r_k^{\rm det}}
\leq 
\begin{cases}
(1+\lambda_k)^{d/2}\lambda^{-d/2}(\log\lambda)^{11/2},
& 1\le q\le \frac{d+1}{4},\\[2mm]
(1+\lambda_k)^{\frac{d-2}{2}+\frac{d+1}{4q}}
\lambda^{-d/2}(\log\lambda)^{11/2},
& \frac{d+1}{4}\le q\le \frac{d+1}{2},\\[2mm]
(1+\lambda_k)^{(d-1)/2}\lambda^{-d/2}(\log\lambda)^{11/2},
& \frac{d+1}{2}\le q\le \infty.
\end{cases}
\end{align}

We now compare $\Omega$ and $\Omega^{\rm det}$. From Theorem \ref{thm:main_intro}, it follows that for $z\in\Omega$ and $|z|\leq \lambda^{2-\epsilon}$.


$$
|z|^{1/2}(1+|z|)^{-\sigma(2q)}
\leq C_N K \lambda^{-\frac{d}{2q}-c\epsilon}(\ln\lambda)^{11/2}\|v(\lambda)\|_{\ell^{2q}}.
$$

\subsection{Further remarks on Theorem \ref{thm:main_intro}}

We also make the following remarks on~Theorem~\ref{thm:main_intro}.

\begin{enumerate}
\item For comparison of the compact manifold bounds with the Euclidean bounds of Frank \cite{MR2820160}, we refer to \cite[Section 4]{Cuenin2026}.

\item The trivial bound
\[
\spec(-\Delta_g+V)
\subset
\{z\in\mathbb{C}:\dist(z,\spec(-\Delta_g))\leq \|V\|_{L^\infty}\}
\]
implies that $\spec(H_{\omega,\lambda})$ lies in the $\|v(\lambda)\|_{\ell^\infty}$-neighbourhood of $\{\lambda_k^2\}_{k=0}^\infty$. Since $\ell^q \subset \ell^\infty$, this gives the asymptotically (as $k\to\infty$) weaker bound $r_k \leq \|v(\lambda)\|_{\ell^q}$. 

\item The deterministic bounds of \cite{Cuenin2026} are sharp in both the low-energy and high-energy regimes. The former follows from a scaling argument when $M$ is a torus, while the latter follows from the optimality of Sogge's spectral cluster bounds together with an operator-valued Rouch\'e theorem (Gohberg--Sigal theory) when $M$ is the sphere, or more generally a Zoll manifold. It would be interesting to know whether Theorem \ref{thm:main_intro} is sharp up to the logarithmic loss. Due to the probabilistic nature of the bound, this seems to be a difficult question.

\item For additional background on deterministic and random eigenvalue bounds for Schrödinger operators with complex potentials, we refer to the survey \cite{Stefanescu2026}, which also discusses fractional Laplacians.

\item Theorem \ref{thm:main_intro} remains nontrivial even when $v_j(\lambda) \in \mathbb{R}$, so that $H_{\omega,\lambda}$ is self-adjoint; in this case, it gives a high-probability improvement over the trivial $L^\infty$ spectral inclusion discussed in \textup{(3)}.

\end{enumerate}


\section{Tools from probability theory}
\label{s:abstracttools}

We collect some tools from probability theory that were used in \cite{CueninMerz2025}, and that will also be important here to prove Theorems~\ref{thm:abstract1_intro} and~\ref{thm:abstract2_intro}.

\subsection{Sub-Gaussian random variables and tail bounds}
We recall that a (complex) scalar random variable $X$ is called \emph{sub-Gaussian} if it has finite sub-Gaussian norm,
\begin{align*}
  \|X\|_{\psi_2}=\inf\{t>0:\,\mathbf{E}\exp(|X|^2/t^2)\leq 2\}<\infty.
\end{align*}
We will use the following properties of sub-Gaussian (e.g., centred normalised Gaussian or symmetric Bernoulli) random variables; see, e.g., \cite[Proposition~2.6.1 and Exercise~2.5.10]{Vershynin2018}.

\begin{proposition}[{\cite[Proposition~21]{CueninMerz2025}}]
  \label{prop. properties of subgaussian rv}
  Assume that $(X_j)_{j=1}^{N}$, $N\geq 2$, is a finite collection of i.i.d.~mean-zero sub-Gaussian random variables. 
  \begin{itemize}
  \item[(i)] Then $\sum_{j=1}^{N} X_j$ is also sub-Gaussian and
    \begin{align*}
      \|\sum_{j=1}^{N} X_j\|_{\psi_2}^2
      \lesssim \sum_{j=1}^{N}\|X_j\|_{\psi_2}^2.
    \end{align*} 
  \item[(ii)] We have
    \begin{align}
      \label{eq:dudley}
      \mathbf{E}\max_{j\leq N}|X_j|
      \lesssim \sqrt{\log N}\max_{j\leq N}\|X_j\|_{\psi_2}.
    \end{align}
  \end{itemize}
\end{proposition}
We call \eqref{eq:dudley} Dudley's inequality, cf.~\cite[Section~8.1]{Vershynin2018}.

\smallskip
We now consider tail bounds for (infinite-dimensional) \emph{vector-valued} normalised Gaussian or Bernoulli random variables $X$ which have finite $\psi_2$-norm. We have $(\mathbf{E}\|X\|^p)^{1/p}\asymp(\mathbf{E}\|X\|^q)^{1/q}$ for all $p,q>0$ (cf.~\cite[Corollary 3.2 and Theorem 4.7]{LedouxTalagrand1991}), which, combined with \cite[(3.5), (4.12)]{LedouxTalagrand1991} implies
\begin{align*}
  \mathbf{P}(\|X\| > t)
  \leq \exp\left(-\frac{ct^2}{(\mathbf{E}\|X\|)^2}\right)
\end{align*}
for some $c>0$. Thus, we have the following estimate.

\begin{lemma}[{\cite[Lemma~22]{CueninMerz2025}}]\label{lem:tailbound}
  Let $X$ be a vector-valued Gaussian or Bernoulli random variable.
  If $\mathbf{E}\|X\|\leq C$, then
  \begin{align}
    \mathbf{P}(\|X\| > KC)
    \leq \exp(-cK^2)
  \end{align}
 for any $K>0$.
\end{lemma}
We are not aware of generalisations of this lemma to more general sub-Gaussian vector-valued random variables.

\subsection{Entropy bounds}
We state an abstract version of the chaining argument \cite[Corollary~14]{CueninMerz2025}, whose proof was inspired by that of Bourgain \cite{Bourgain2002}. The proof is the same.

\begin{proposition}\label{prop. abstract chaining}
  \label{prop. representaion}
  Let $\mathcal{H}$ be a finite-dimensional Hilbert space, $E$ a nonempty finite set and
  $S:\mathcal{H}\to \ell^{\infty}(E)$ a linear operator.
  For every $k\in \N$, there exist sets $\mathcal{F}_k\subset \ell^{\infty}(E)$ with the following properties.
  \begin{enumerate}
  \item[\rm(a)] $\log\#\mathcal{F}_k\lesssim 4^k \log \# E$.

  \item[\rm(b)] For every $\xi\in \mathcal{F}_k$, we have
    \begin{align}
      \label{eq:boundxiinfty}
      \|\xi\|_{\ell^{\infty}(E)}\lesssim 2^{-k}\|S\|_{\mathcal{H}\to \ell^{\infty}(E)}.
    \end{align}

  \item[\rm(c)] For each $g\in \mathcal{H}$ with $\|g\|_{\mathcal{H}}\leq 1$ there is a representation
    \begin{align}
      Sg = \sum_{k\in \N}\xi_k\quad \text{for some}\quad \xi_k\in\mathcal{F}_k.
    \end{align} 
    Moreover, each $\xi_k$ lies in the image of the unit ball in $\mathcal{H}$ under $S$.
  \end{enumerate}
\end{proposition}

\section{Abstract results}

In this section, we first prove the discrete abstract estimate (Theorem \ref{thm:abstract1_intro}). We then introduce the appropriate setup for the continuous abstract estimate (Theorem \ref{thm:abstract2_intro}) and prove a rigorous version of it  (Theorem \ref{thm:abstract2}). The results of this section form the technical core of the paper and are of independent interest.

\subsection{Discrete model} We now prove Theorem \ref{thm:abstract1_intro}.
\begin{proof} 1. We prove this for $q<\infty$; the case $q=\infty$ is analogous. We assume first that $\mathcal{H}_1,\mathcal{H}_2$ are finite-dimensional.
 We apply Proposition~\ref{prop. abstract chaining} to both $S_1$ and $S_2$ and denote the resulting sets by $\mathcal{F}_{1,k}$ and $\mathcal{F}_{2,\ell}$, respectively. Then
 \begin{align*}
\|S_1^*v_{\omega}S_2\|_{\mathcal{H}_2\to \mathcal{H}_1}
=\sup_{\|g_1\|_{\mathcal{H}_1}=\|g_2\|_{\mathcal{H}_2}=1}|\langle S_1g_1,v_{\omega}S_2g_2\rangle|
\leq \sum_{k,\ell\in \N}\max_{(\xi_1,\xi_2)\in \mathcal{F}_{1,k}\times \mathcal{F}_{2,\ell}}|\langle\xi_1,v_{\omega}\xi_2\rangle|,
 \end{align*}
 where $\langle\cdot,\cdot\rangle$ is the inner product in $\ell^2(E)$. 
 By monotonicity of the expectation,
 \begin{align*}
\mathbf{E}\|S_1^*v_{\omega}S_2\|_{\mathcal{H}_2\to \mathcal{H}_1} \leq \sum_{k,\ell\in \N}\mathbf{E}\max_{(\xi_1,\xi_2)\in \mathcal{F}_{1,k}\times \mathcal{F}_{2,\ell}}|\langle\xi_1,v_{\omega}\xi_2\rangle|.   
 \end{align*}
 Note that, by definition of $v_{\omega}$ in \eqref{def:v_omega discrte} and Proposition \ref{prop. properties of subgaussian rv} (i), we have
\begin{align}\label{eq:langle_xi1_vomega_xi2_rangle}
\|\langle\xi_1,v_{\omega}\xi_2\rangle\|_{\psi_2}
 =\Big\|\sum_{j\in J}\omega_j\langle \xi_1,v\xi_2\rangle_{\ell^2(E_j)}\Big\|_{\psi_2}\lesssim \Big(\sum_{j\in J}|\langle \xi_1,v_j\xi_2\rangle_{\ell^2(E_j)}|^2\Big)^{1/2},
 \end{align}
 where we have also used  $\sup_{j\in J}\|\omega_j\|_{\psi_2}<\infty$.
By Proposition~\ref{prop. abstract chaining}~(a), the cardinality $N$ of the set $\mathcal{F}_{1,k}\times \mathcal{F}_{2,\ell}$ satisfies
\begin{align}\label{eq:logN}
\log N=
    \log\# (\mathcal{F}_{1,k}\times \mathcal{F}_{2,\ell})
    \lesssim (4^k+4^{\ell})\, \log \# E.
\end{align}
By Dudley's inequality \eqref{eq:dudley},
\eqref{eq:langle_xi1_vomega_xi2_rangle}, \eqref{eq:logN}, disjointness of the sets $E_j$, and H\"older's inequality for sums, \eqref{eq:boundxiinfty},
\begin{align*}
  & \mathbf{E}\max_{(\xi_1,\xi_2)\in \mathcal{F}_{1,k}\times \mathcal{F}_{2,\ell}}| \langle\xi_1,v_{\omega}\xi_2\rangle|
  \lesssim \sqrt{\log N}\Big(\sum_{j\in J}|\langle \xi_1,v_j\xi_2\rangle_{\ell^2(E_j)}|^2\Big)^{\frac{1}{2}} \\
  & \leq \max_{(\xi_1,\xi_2)\in \mathcal{F}_{1,k}\times \mathcal{F}_{2,\ell}} \sqrt{\log N}\Big(\sum_{j\in J}\| \xi_1\|_{\ell^{p'}(E_j)}^2\|v\xi_2\|_{\ell^{p}(E_j)}^2\Big)^{\frac{1}{2}} \\
  & \leq \max_{(\xi_1,\xi_2)\in \mathcal{F}_{1,k}\times \mathcal{F}_{2,\ell}} \sqrt{\log N}\Big(\sum_{j\in J}\| \xi_1\|_{\ell^{p'}(E_j)}^2\|\xi_2\|_{\ell^{\infty}(E_j)}^2\|v\|_{\ell^{p}(E_j)}^2\Big)^{\frac{1}{2}} \\
  & \leq \max_{(\xi_1,\xi_2)\in \mathcal{F}_{1,k}\times \mathcal{F}_{2,\ell}} \sqrt{\log N}\|\xi_1\|_{\ell^{p'}(E)}\|\xi_2\|_{\ell^{\infty}(E)}\Big(\sum_{j\in J}\|v\|_{\ell^{p}(E_j)}^{2q}\Big)^{\frac{1}{2q}} \\
  & \lesssim \sqrt{\log \# E}(2^k+2^{\ell})2^{-\ell}\|S_1\|_{\mathcal{H}_1\to \ell^{p'}(E)}\|S_2\|_{\mathcal{H}_2\to \ell^{\infty}(E)}\Big(\sum_{j\in J}\|v\|_{\ell^p(E_j)}^{2q}\Big)^{\frac{1}{2q}}.
\end{align*}
Interchanging the roles of $\xi_1,\xi_2$, we can replace $(2^k+2^{\ell})2^{-\ell}$ by $(2^k+2^{\ell})2^{-k}$, so that 
\begin{align*}
  & \mathbf{E}\max_{(\xi_1,\xi_2)\in \mathcal{F}_{1,k}\times \mathcal{F}_{2,\ell}}| \langle\xi_1,v_{\omega}\xi_2\rangle| \\
  & \quad \lesssim \sqrt{\log \# E}\|S_1\|_{\mathcal{H}_1\to \ell^{p'}(E)}\|S_2\|_{\mathcal{H}_2\to \ell^{\infty}(E)}\Big(\sum_{j\in J}\|v\|_{\ell^p(E_j)}^{2q}\Big)^{\frac{1}{2q}} =: A.
\end{align*}
Combining this with the trivial bound
\begin{align*}
\max_{(\xi_1,\xi_2)\in \mathcal{F}_{1,k}\times \mathcal{F}_{2,\ell}}|\langle\xi_1,v_{\omega}\xi_2\rangle|\lesssim 2^{-k-\ell}\|S_1\|_{\mathcal{H}_1\to \ell^{\infty}(E)}\|S_2\|_{\mathcal{H}_2\to \ell^{\infty}(E)}\|v(\lambda)\|_{\ell^1(E)}=:2^{-k-\ell}B,
\end{align*}
we get
\begin{align*}
\mathbf{E}\|S_1^*v_{\omega}S_2\|_{\mathcal{H}_2\to \mathcal{H}_1}
&\lesssim \sum_{k,\ell\in \N}\min(A,2^{-k-\ell}B)
\lesssim A(1+(\log B/A)^2).
\end{align*}
By H\"older's inequality and the inclusion $\ell^{s}(E)\subset\ell^{r}(E)$, for $s\leq r$, we have  
\begin{align*}
  B
  \leq \|S_1\|_{\mathcal{H}_1\to \ell^{p'}(E)}\|S_2\|_{\mathcal{H}_2\to \ell^{\infty}(E)}
  (\# E)^{\frac{1}{p'}}\Big(\sum_{j\in J}\|v\|_{\ell^p(E_j)}^{2q}\Big)^{\frac{1}{2q}}
\end{align*}
so that $\log(B/A)\leq \frac{1}{p'}\log\#E$.
This concludes the proof for the case where $\mathcal{H}_1,\mathcal{H}_2$ are finite-dimensional.

2. In the case where $\mathcal{H}_1,\mathcal{H}_2$ are infinite-dimensional and separable, let 
\begin{align*}
    \mathcal{H}_i^{(1)}\subset \mathcal{H}_i^{(2)}\subset\ldots\subset \mathcal{H}_i,\quad i=1,2,
\end{align*}
be an increasing sequence of finite-dimensional subspaces whose union is dense in $\mathcal{H}_i$. Let $P_i^{(k)}$ be the orthogonal projections onto $\mathcal{H}_i^{(k)}$. For each $\omega=\{\omega_j\}_{j\in J}$ and $k\in\N$ we set
\begin{align*}
N_k(\omega):=\|P_1^{(k)}S_1^*v_{\omega}S_2P_2^{(k)}\|_{\mathcal{H}_2^{(k)}\to \mathcal{H}_1^{(k)}}.
\end{align*}
For fixed $\omega$, the numbers $N_k(\omega)$ are nondecreasing in $k$ and 
\begin{align*}
\lim_{k\to\infty}N_k(\omega)=\sup_{k\in\N} N_k(\omega)=\|S_1^*v_{\omega}S_2\|_{\mathcal{H}_2\to \mathcal{H}_1}.
\end{align*}
By the monotone convergence theorem,
\begin{align*}
\mathbf{E}\|S_1^*v_{\omega}S_2\|_{\mathcal{H}_2\to \mathcal{H}_1}=\lim_{k\to\infty}\mathbf{E}N_k(\omega).
\end{align*}
Since $\mathbf{E}N_k(\omega)$ is uniformly bounded in $k$ by the right-hand side of \eqref{eq:abstract1}, the general case is proved.
\end{proof}

\subsection{Continuous model}\label{sec:abstract}
We now introduce the setup for the continuous model. The main result of this subsection is Theorem \ref{thm:abstract2}.

\begin{definition}
\label{defdef: alpha regular measure}
  Let $(X,\rho,\mu)$ be a metric measure space and $d>0$. We say that $X$ is \emph{Ahlfors $d$-regular} if $\mu$ is a Radon measure and there is $C_A>0$ such that
  \begin{align}\label{def: alpha regular measure}
    C_A^{-1}r^{d} 
    \leq \mu(B(x,r))
    \leq C_A r^{d}
\end{align}
  for all balls $B(x,r)\subset X$ with $r\in (0,1)$. 
  We say $X$ is \emph{upper/lower Ahlfors $d$-regular} if the upper/lower bound in \eqref{def: alpha regular measure} holds.
\end{definition}

Clearly, if $X$ is upper Ahlfors $d$-regular, then for all $p\in [1,\infty]$,
\begin{align}\label{eq: abstract Hoelder}
  \|f\|_{L^{p}(B(x,r))}\leq C_A^{1/p}r^{d/p}\|f\|_{L^{\infty}(B(x,r))}.    
\end{align}

Note furthermore that if $X$ is Ahlfors $d$-regular, then $X$ is $\sigma$-finite.

\begin{definition}
Let $(X,\rho,\mu)$ be a metric measure space and $r>0$.
We say that a collection of measurable functions $\mathcal{F}$ on $X$ satisfies the \emph{local constancy property at scale $r$} if there exists a constant $C_{\rm loc}$ such that for all $f\in\mathcal{F}$ and all balls $B\subset X$ with radii at most $2r$, we have 
\begin{align}\label{eq. local constancy property}
 \|f\|_{L^{\infty}(B)}\leq C_{\rm loc}\int_{X}|f|w_B\rd\mu,
\end{align}
where the weights $w_B$ are given by
\begin{align*}
    w_B(x):=\mu(B)^{-1}(1+r^{-1}\dist(x,B))^{-100d}.
\end{align*}
\end{definition}

\begin{remark}\label{rem:locallyconstant}
  (1) By definition, if $\mathcal{F}$ satisfies the local constancy property at scale $r$, then the same is true for all scales $\leq r$.

  \noindent
  (2) As an example, for $X=\R^d$, $\rho$ the Euclidean metric, and $\mu$ the $d$-dimensional Lebesgue measure, the family of functions $\{f\in \mathcal{S}(\R^d) \,|\, \exists a\in\R^d,b>0:\,\supp\hat f \subseteq B_a(b/r)\}$ satisfies the local constancy property at scale $r$.
\end{remark}

\begin{definition}
 Let $(X,\rho)$ be a metric space and $r>0$. We call $\Lambda\subset X$ an {\em $r$-separated set} if
\begin{align*}
    \rho(x,x')>r\quad\mbox{for all distinct }x,x'\in \Lambda.
\end{align*}
We call $\Lambda$ {\em maximal $r$-separated} if $\Lambda$ is $r$-separated and if for every $x\in X\setminus\Lambda$ there exists $x'\in\Lambda$ such that $\rho(x,x')\leq r$.
\end{definition}

\begin{definition}
Let $(X,\rho,\mu)$ be a metric measure space and $r>0$.
 We say that a collection of measurable functions $\mathcal{F}$ on $X$ satisfies the \emph{weak local constancy property at scale $r$} if for every $s\in [1,\infty]$ there exists a constant $C_s$ such that for all $r$-separated sets $\Lambda\subset X$ and $f\in \mathcal{F}$, we have 
\begin{align}\label{eq:abstract comparison ellp Lp} 
\Big(\sum_{x\in\Lambda}|f(x)|^s\Big)^{1/s}\leq C_sr^{-d/s}\Big(\int_X|f|^s
\rd\mu\Big)^{1/s}.
\end{align}
\end{definition}

\begin{proposition}\label{cor. abstract comparison ellp Lp}
  Let $X$ be Ahlfors $d$-regular, $r>0$, and $\Lambda\subset X$ be an $r$-separated set. Let $T:\mathcal{H}\to L^{p}(X)$ be a bounded linear operator for some $p\in [1,\infty]$. Assume $\Ran T$ has the weak local constancy property at scale $r$. Then for $S_{\Lambda}:\mathcal{H}\to \ell^{p}(\Lambda)$ defined by $S_{\Lambda}f(x):=Tf(x)$, $x\in\Lambda$, we have
  \begin{align*}
    \|S_{\Lambda}\|_{\mathcal{H}\to \ell^{p}(\Lambda)}\leq C_pr^{-d/p}\|T\|_{\mathcal{H}\to L^{p}(X)}.
  \end{align*}
\end{proposition}

\begin{proof}
  This follows immediately from the definition.
 \end{proof}

\begin{proposition}\label{prop. abstract comparison ellp Lp}
The local constancy property (at scale $r$) implies the weak local constancy property (at scale $r$).
\end{proposition}

\begin{proof}
  Assume that $\mathcal{F}$ has the local constancy property at scale $r$.
  Let $\Lambda\subset X$ be an $r$-separated set and $f\in \mathcal{F}$. By \eqref{eq. local constancy property} and H\"older's inequality,
  \begin{align*}
    \sum_{x\in\Lambda}|f(x)|^s
    & \leq \sum_{x\in\Lambda}\|f\|_{L^{\infty}(B(x,r))}^s    
      \leq C_{\rm loc}^s \sum_{x\in\Lambda}\Big(\int_{X}|f|w_{B(x,r)}\rd\mu\Big)^s \\
   & \leq C_{\rm loc}^s \sum_{x\in\Lambda} \Big(\int_{X}|f|^sw_{B(x,r)}\rd\mu\Big)\Big(\int_X w_{B(x,r)}\rd\mu\Big)^{s-1} \\
    &\leq C_{\rm loc}^s (C_A2^{2d+2})^{s-1}\sum_{x\in\Lambda}\int_{X}|f|^sw_{B(x,r)}\rd\mu,
  \end{align*}
  where we have also used both the upper and lower bounds in \eqref{def: alpha regular measure} to estimate
  \begin{align*}
    \int_X w_B\rd \mu
    & \leq r^{-d}\big(\mu(\{\dist(y,B)\leq r\})+\sum_{k=0}^{\infty}2^{-100dk}\mu(\{r2^k\leq\dist(y,B)\leq r2^{k+1}\})\big) \\
    & \leq C_A\big(2^d+\sum_{k=0}^{\infty}(1+2^{k+1})^d2^{-100dk}\big)\leq 2^{2d+2}C_A.
  \end{align*}
  Since Ahlfors regularity implies $\sigma$-finiteness of $X$, Tonelli's theorem  and a similar calculation as above yield
  \begin{align*}
    \sum_{x\in\Lambda}\int_{X}|f|^sw_{B(x,r)}\rd\mu
    & = \int_{X}|f|^s\sum_{x\in\Lambda}w_{B(x,r)} \rd\mu\\
    & \leq r^{-d}\int_{X}|f|^s\sum_{x\in\Lambda}\big(\mathbf{1}_{B(x,2r)}+\sum_{k=0}^{\infty}2^{-100dk}\mathbf{1}_{B(x,2^{k+2}r)}\big)\rd\mu \\
    & \leq C_A^2\big(5^d+\sum_{k=0}^{\infty}2^{-100dk}2^{(k+4)d}\big)r^{-d}\int_X |f|^s\rd\mu\\
    & \leq 2^{4d+2}C_A^2r^{-d}\int_X |f|^s\rd\mu,
  \end{align*}
  from which \eqref{eq:abstract comparison ellp Lp} follows. In the penultimate inequality, we have used that for $M\geq 1$,
  \begin{align}\label{eq: abstract overlap bound}
    \sum_{x\in\Lambda}\mathbf{1}_{B(x,Mr)}\leq C_A^2(2M+1)^d.
  \end{align}
  Indeed, for fixed $y\in X$, consider the set
  \begin{align*}
    \Lambda_y := \{x\in \Lambda:y\in B(x,Mr)\}.
  \end{align*}
  Since $\Lambda$ is $r$-separated, the balls $\{B(x,r/2)\}_{x\in\Lambda}$ are disjoint. By the triangle inequality,
  \begin{align*}
    \bigcup_{x\in\Lambda_y}B(x,r/2)\subset B(y,(M+1/2)r),
  \end{align*}
  so that, by \eqref{def: alpha regular measure} (both upper and lower bounds), 
  \begin{align*}
    C_A^{-1}2^{-d}r^d\#\Lambda_y
    \leq \sum_{x\in\Lambda_y}\mu(B(x,r/2))
    \leq \mu(B(y,(M+1/2)r))
    \leq C_A(M+1/2)^dr^d.
  \end{align*}
  Since $y$ was arbitrary, this proves \eqref{eq: abstract overlap bound}. 
\end{proof}

The following development merely uses the weak local constancy property. In applications, the local constancy property can be verified rather directly; see Appendix~\ref{s:localconstancy} below in the case of the range of the spectral projector for $-\Delta_g$.

\begin{definition}
  Given a maximal $r$-separated subset $\Lambda$ of a metric space $(X,\rho)$, we denote by $\mathcal{V}_x$ the {\em Voronoi cell} associated to $x\in\Lambda$,
  \begin{align}
    \mathcal{V}_x := \{y\in X:\rho(x,y)<\rho(x',y)\mbox{ for all }x'\in\Lambda\setminus\{x\}\}.
  \end{align}
\end{definition}

All Voronoi cells are mutually disjoint and, if $X$ is  Ahlfors $d$-regular, $\{\mathcal{V}_x\}_{x\in\Lambda}$ covers $X$ up to a set of measure zero. Thus, for any integrable function $f:X\to \C$,
\begin{align*}
  \int_X f\rd\mu=\sum_{x\in\Lambda}\int_{\mathcal{V}_x}f\rd\mu.
\end{align*}

\begin{lemma}\label{lemma: radius Voronoi}
  Let $(X,\rho)$ be a metric space, $r>0$, and $\Lambda\subseteq X$ be a maximal $r$-separated subset. Then $B(x,r/2)\subset \mathcal{V}_x\subset B(x,2r)$ for all $x\in\Lambda$.
\end{lemma}

\begin{proof}
  $B(x,r/2)\subset \mathcal{V}_x$:  Let $y\in B(x,r/2)$ and $x'\in\Lambda\setminus\{x\}$. Then 
  \begin{align*}
    \rho(x',y)\geq \rho(x',x)-\rho(x,y)\geq r-\frac{r}{2}=\frac{r}{2},
  \end{align*}
  which implies $\rho(x,y)< r/2 \leq \rho(x',y)$ and hence $y \in\mathcal{V}_{x}$.

 \smallskip
 $\mathcal{V}_x\subset B(x,2r)$: Let $y\in \mathcal{V}_x$. Without loss of generality, we may assume $y\neq x$, otherwise there is nothing to prove. Assume for contradiction that $\rho(x,y)\geq 2r$. Since $\Lambda$ is an $r$-net (see e.g. \cite[Lemma~4.2.6]{Vershynin2018}), there exists $x'\in\Lambda\setminus\{x\}$ such that $\rho(x',y)\leq r$. But this implies
 \begin{align*}
     r\geq \rho(x',y)>\rho(x,y)\geq 2r,
 \end{align*}
 a contradiction.
\end{proof}

\begin{definition}\label{asspt: uniform Radon-Nikodym}
Let $(X,\rho,\mu)$ be Ahlfors $d$-regular, $r>0$, and
let $\Lambda\subset X$ be an $r$-separated set. 
\begin{enumerate}
    \item We say that the pair $(X,\Lambda)$ satisfies the \emph{uniform reparametrisation property at scale $r$}
if there exists a finite measure space $(Y,\nu)$ with $\nu(Y)\leq C_Yr^d$ and, for each $x\in\Lambda$, a measurable bijection 
$\phi_x : X\supset B(x,2r) \to Y$ with measurable inverse,
such that the pushforward measures $\mu_x:=(\phi_x)_*(\mu|_{B(x,2r)})$ 
are absolutely continuous with respect to $\nu$, and the Radon–Nikodym derivatives
$w_x = \frac{d\mu_x}{d\nu}$ satisfy the uniform bounds
\begin{align}\label{eq: uniform bounds on Radon–Nikodym derivatives}
0\leq w_x(y) \leq C_{\rm RN}\qquad \text{for all $x\in\Lambda$ and $\nu$-a.e.\ $y\in Y$}.
\end{align}
\item We say that $(X,\Lambda)$ satisfies the \emph{strong uniform reparametrisation property at scale $r$} if there exists a separable metric measure space $(Y,\delta,\nu)$ with finite Ahlfors $d$-regular measure $\nu$ satisfying $C_Y^{-1}r^d \leq \nu(Y) \leq C_Yr^d$ and, for each $x\in\Lambda$, a bi-measurable bijection $\phi_x : X\supset B(x,2r) \to Y$ such that
$\phi_x^{-1}$ is Lipschitz with uniform Lipschitz bounds in $x$, i.e., there exists a constant $C_{L}$ such that 
\begin{align*}
   \rho(\phi_x^{-1}(y),\phi_x^{-1}(z))\leq C_{L} \delta(y,z) \quad\mbox{for all }y,z\in Y.
\end{align*} 
\item We say that $X$ satisfies the \emph{(strong) uniform reparametrisation property at scales $\geq r$} if for every $r'\geq r$ and every $r'$-separated set $\Lambda\subset X$, the pair $(X,\Lambda)$ satisfies the (strong) uniform reparametrisation property at scale $r'$.
\end{enumerate}
\end{definition}

\begin{remark}
(1) In the uniform reparametrisation property, the upper bound
\(\nu(Y)\leq C_Y r^d\) is in fact complemented by a lower bound of
the same order. Indeed, by Ahlfors \(d\)-regularity of \(X\), for each
\(x\in\Lambda\),
\[
    \mu(B(x,2r)) \geq C_A^{-1} (2r)^d,
\]
while the Radon--Nikodym bound gives
\[
    \mu(B(x,2r))
    = \int_Y w_x\,d\nu
    \leq C_{\rm RN}\nu(Y).
\]
Hence
\[
    \nu(Y)\geq 2^dC_A^{-1}C_{\rm RN}^{-1}r^d.
\]

\noindent
(2) The uniform reparametrisation property is a purely
measure-theoretic condition: it provides a common measurable model for
the balls \(B(x,2r)\), together with uniform control of the associated
Radon--Nikodym densities. The strong uniform reparametrisation property is a strengthening, in which the common model is a
metric measure space and the Radon--Nikodym condition is replaced by a uniform
Lipschitz condition. The main abstract theorem in this subsection (Theorem \ref{thm:abstract2}) only uses the weaker assumption. However, in most applications, $X$ will satisfy the stronger assumption (see Remark \ref{rem:examples reparametrisation}).
Moreover, the strong assumption is usually easier to verify. This is the reason we included it in the above definition. 
\end{remark}

\begin{remark}[Examples]\label{rem:examples reparametrisation}
The strong uniform reparametrisation properties are satisfied in a number of
standard situations.

\begin{enumerate}
    \item Let \(X=\mathbb R^d\), $\rho$ be the Euclidean metric, and $\mu$ be the $d$-dimensional Lebesgue
    measure \(\mathcal L^d\). For any \(r>0\) and any \(r\)-separated set
    \(\Lambda\subset\mathbb R^d\), one may take
    \[
        Y=B(0,2r),\qquad \nu=\mathcal L^d|_{B(0,2r)},
        \qquad \phi_x(y)=y-x .
    \]
    Then \(\phi_x:B(x,2r)\to Y\) is an isometry and
    \((\phi_x)_*(\mathcal L^d|_{B(x,2r)})=\nu\). Thus \(w_x\equiv 1\),
    and the strong uniform reparametrisation property holds at all
    scales. The same argument applies to finite-dimensional normed
    spaces with Lebesgue measure.

    \item More generally, if \(X\) is a Carnot group equipped with a
    homogeneous distance and Haar measure, then left translation gives
    the required parametrisations. Namely, for \(Y=B(e,2r)\) and
    \[
        \phi_x(y)=x^{-1}y ,
    \]
    the maps \(\phi_x:B(x,2r)\to Y\) are isometries and preserve Haar
    measure. Hence again \(w_x\equiv 1\), and the strong uniform reparametrisation property holds at all scales.

    \item Let \(M\) be a (not necessarily closed) \(d\)-dimensional Riemannian manifold of
    bounded geometry, equipped with its geodesic distance and Riemannian
    volume measure. At sufficiently small scales, one may use normal
    coordinates centred at \(x\) to show that \(M\) satisfies the strong uniform reparametrisation property locally. In particular, any closed manifold \(M\) with smooth Riemannian metric (as considered in this article) satisfies the strong uniform reparametrisation property (see the proof of Proposition \ref{prop: manifold spectral cluster expectation} for details).

\end{enumerate}
\end{remark}

\begin{lemma}\label{lemma uniform bilipschitz}
The strong uniform reparametrisation property implies the uniform reparametrisation property.
\end{lemma}

\begin{proof}
1. By the Lipschitz assumption, for any ball $B(z,t)$ in $Y$, we see
$\phi_x^{-1}(B(z,t))\subset B(\phi_x^{-1}(z),C_Lt)$. Thus, by definition of the pushforward measure, the Ahlfors regularity of $\mu$ and lower Ahlfors regularity of $\nu$,
\begin{align*}
  \mu_x(B(z,t))
  = \mu(\phi_x^{-1}(B(z,t)))
  \leq \mu(B(\phi_x^{-1}(z),C_Lt))
  \leq C_A(C_Lt)^d
  \leq C_A C_Y^2 C_L^d\nu(B(z,t)).
\end{align*}
2. Let $U\subset Y$ be an open set, and for each $y\in U$ choose $\delta_y\in (0,1)$ such that $B(y,\delta_y)\subset U$. By the $5r$-covering theorem (see, e.g., \cite[Theorem 2.1]{Mattila1995}), there exists a finite or countable sequence of disjoint balls $B_i=B(y_i,\delta_{y_i})$ such that the collection $\{5B_i\}$ covers $U$. Thus, by the first step,
\begin{align*}
    \mu_x(U)\leq \sum_i\mu_x(5B_i)
    \leq C_A C_Y^2 C_L^d \sum_i\nu(5B_i)
    \leq C_A^2 C_Y^2 (5C_L)^d \sum_i\nu(B_i)
    \leq C_A^2 C_Y^2 (5C_L)^d\nu(U).
\end{align*}
3. Let $E\subset Y$ be a measurable set with $\nu(E)=0$, and let $\epsilon>0$. By outer regularity of $\nu$, there exists an open set $U\subset Y$ such that $E\subset U$ and $\nu(U)\leq \epsilon$. By the second step, it follows that $\mu_x(U) \leq  C_A^2 C_Y^2 (5C_L)^d\epsilon$. Since $\epsilon$ was arbitrary, we have $\mu_x(E)=0$, so that $\mu_x\ll\nu$.\\

\noindent 
4. By the Radon-Nikodym theorem and the estimate in the first step, for any ball $B\subset Y$, we have
\begin{align*}
    \mu_x(B)
    = \int_Bw_x\rd\nu
    \leq C_A C_Y^2 C_L^d\nu(B)
    \implies \int_B (w_x- C_A C_Y^2 C_L^d)\rd\nu\leq 0.
\end{align*}
Consider the sets 
\begin{align*}
    E:=\{y\in Y:w_x- C_A C_Y^2 C_L^d>0\},\quad 
    E_n:=\{y\in Y:w_x-  C_A C_Y^2 C_L^d \geq 1/n\},
\end{align*}
and $E_{n,k}:=E_n\cap B(y_0,k)$ for some fixed $y_0\in Y$. Then $E_{n,k}\nearrow_{k\to\infty}
E_n$ and $E_n\nearrow_{n\to\infty} E$. Since
\begin{align*}
  0 \leq \nu(E_{n,k})
  \leq n\int_{B(y_0,k)} (w_x- C_A C_Y^2 C_L^d)\rd\nu\leq 0,
\end{align*}
it follows that $\nu(E_n)=\lim_{k\to\infty}\nu(E_{n,k})=0$ and $\nu(E)=\lim_{n\to\infty}\nu(E_n)=0$. Hence, $w_x \leq C_A C_Y^2 C_L^d$ almost everywhere.
\end{proof}

\begin{lemma}\label{lem: Fubini reparametrisation}
  Assume that $(X,\Lambda)$ satisfies the uniform reparametrisation property at scale $r$. Then for every $\mu$-integrable function $f:X\to\C$,
  \begin{align}\label{eq: change of variables formula}
    \sum_{x\in\Lambda} \int_{B(x,2r)} f \, d\mu
    = \int_{Y} \sum_{x\in\Lambda} f\circ \phi_x^{-1} \, w_x\, \rd\nu.
  \end{align}
\end{lemma}

\begin{proof}
  This follows from the change-of-variables formula, the Radon–Nikodym theorem and Fubini's theorem. More precisely, by change-of-variables (definition of pushforward) and the Radon–Nikodym theorem, we have the identity
  \[
    \int_{B(x,2r)} f\,\rd\mu
    = \int_{Y} f\circ \phi_x^{-1} \, w_x\, \rd\nu.
  \]
  To justify Fubini's theorem, we apply the same argument to $|f|$, but in the reverse direction,
  \[
    \sum_{x\in\Lambda}\int_{Y} |f|\circ \phi_x^{-1} \, w_x\, \rd\nu
    = \sum_{x\in\Lambda}\int_{B(x,2r)} |f|\,\rd\mu\leq C_A^2 5^d\|f\|_{L^1(X)},
  \]
  where we used the finite overlap bound~\eqref{eq: abstract overlap bound}.
\end{proof}

\begin{corollary}
Assume that $(X,\Lambda)$ satisfies the uniform reparametrisation property at scale $r$ and that $\Lambda$ is maximal $r$-separated. Then for every $\mu$-integrable function $f:X\to\C$,
  \begin{align}\label{eq: change of variables formula 2}
    \int_X f\rd\mu=\sum_{x\in \Lambda}\int_{\mathcal{V}_x}f\rd\mu
    = \sum_{x\in \Lambda}\int_{B(x,2r)}f\mathbf{1}_{\mathcal{V}_x}\rd\mu
    = \int_{Y} \sum_{x\in\Lambda} (f \mathbf{1}_{\mathcal{V}_x})\circ \phi_x^{-1} \, w_x\, d\nu.
  \end{align}
\end{corollary}

\begin{proof}
    The first equality follows from maximality, the second from $\mathcal{V}_x\subset B(x,2r)$ (see Lemma \ref{lemma: radius Voronoi}), and the third from Lemma \ref{lem: Fubini reparametrisation}.
\end{proof}

\begin{theorem}\label{thm:abstract2}
Let $X$ be an Ahlfors $d$-regular metric measure space, let $2\leq q\leq\infty$ and $1\leq p\leq 2\leq p'\leq \infty$ satisfy $\frac{1}{q}=\frac{1}{p}-\frac{1}{p'}$. Let
\[
T_i:\mathcal{H}_i\to L^\infty(X)\cap L^{p'}(X), \qquad i=1,2,
\]
be bounded linear operators, where $\mathcal{H}_i$ are complex separable Hilbert spaces. Assume that $\Ran(T_1)\cup\Ran (T_2)$ has the weak local constancy property at scale $r$.
  Let $\Lambda$ be a finite subset of $X$, let $V\in L^{1}_{\rm loc}(X)$ be a measurable complex-valued function, and let
  \begin{align}
  \label{eq:defrandomisation}
    V_{\omega}(y):=\omega_xV(y)\quad\mbox{for } x\in \Lambda,y\in \mathcal{V}_x,
  \end{align}
  where $\omega_x$ are i.i.d.~symmetric Bernoulli or centred normalised Gaussian random variables. 
  \begin{enumerate}
      \item Assume that $\Lambda$ is maximal $r$-separated and that $(X,\Lambda)$ has the uniform re\-pa\-ra\-me\-tri\-sa\-tion property at scale $r$. Then we have
  \begin{align*}
\mathbf{E}\|T_1^*V_{\omega}T_2\|_{\mathcal{H}_2\to \mathcal{H}_1}
    &\lesssim r^{d/p} (\log\#\Lambda)^{5/2}\|T_1\|_{\mathcal{H}_1\to L^{p'}(X)}\|T_2\|_{\mathcal{H}_2\to L^{\infty}(X)}\\
    &\qquad \times\Big(\sum_{x\in\Lambda}\|V\|_{L^{\infty}(B(x,2r))}^{2q}\Big)^{1/(2q)}.
  \end{align*}    
  \item Assume that $X$ has the uniform reparametrisation property at scales $\geq r$ and that $\Lambda$ is maximal $r'$-separated for some $r'\geq r$. Then we have
  \begin{align*}
\mathbf{E}\|T_1^*V_{\omega}T_2\|_{\mathcal{H}_2\to \mathcal{H}_1}
    &\lesssim r^{d/p} (r'/r)^{\frac{d}{p}-\frac{d}{2q}}(\log\#\Lambda)^{5/2}\|T_1\|_{\mathcal{H}_1\to L^{p'}(X)}\|T_2\|_{\mathcal{H}_2\to L^{\infty}(X)}\\
    &\qquad \times\Big(\sum_{x\in \Lambda}\|V\|_{L^{\infty}(B(x,2r'))}^{2q}\Big)^{1/2q}.
  \end{align*}    
   \end{enumerate}
 The above inequalities are understood to hold with the obvious modification for $q=\infty$.
\end{theorem}

\begin{proof}
(1) For fixed $y\in Y$, we define
  \begin{align}\label{eq: def Siy}
    S_{i,y}:\mathcal{H}_i\to \ell^{\infty}(\Lambda),\quad (S_{i,y}g)(x):=(T_ig)(\phi_x^{-1}(y)),
  \end{align}
  and $v_{\omega,y}(x):=(V_{\omega}\mathbf{1}_{\mathcal{V}_x})(\phi_x^{-1}(y))w_x(y)$ for $x\in\Lambda$.
  Let $g_1\in \mathcal{H}_1,g_2\in \mathcal{H}_2$. By \eqref{eq: change of variables formula 2}, we can write
  \begin{align*}
    \int_X \overline{T_1g_1}T_2g_2V_{\omega}\rd\mu
    = \int_Y\sum_{x\in\Lambda}\overline{(S_{1,y}g_1)(x)}(S_{2,y}g_2)(x)v_{\omega,y}(x)\rd\nu(y). 
  \end{align*}
  Taking the modulus and then the supremum over all unit vectors $g_1,g_2$, it follows that
  \begin{align*}
    \|T_1^*V_{\omega}T_2\|_{\mathcal{H}_2\to \mathcal{H}_1}\leq \int_Y \|S_{1,y}^*v_{\omega,y}S_{2,y}\|_{\mathcal{H}_2\to \mathcal{H}_1}   \rd\nu(y),
  \end{align*}
  and taking expectation,
  \begin{align*}
    \mathbf{E}\|T_1^*V_{\omega}T_2\|_{\mathcal{H}_2\to \mathcal{H}_1}\leq \int_Y \mathbf{E}\|S_{1,y}^*v_{\omega,y}S_{2,y}\|_{\mathcal{H}_2\to \mathcal{H}_1}   \rd\nu(y).
  \end{align*}
  The claimed bound now follows from Theorem~\ref{thm:abstract1_intro}, 
  Proposition~\ref{cor. abstract comparison ellp Lp} and the assumption $\nu(Y)\leq C_Yr^d$. More precisely, we have
  \begin{align*}
    & \int_Y \mathbf{E}\|S_{1,y}^*v_{\omega,y}S_{2,y}\|_{\mathcal{H}_2\to \mathcal{H}_1}   \rd\nu(y)
    \lesssim (\log\#\Lambda)^{5/2}\sup_{y\in Y}\|S_{1,y}\|_{\mathcal{H}_1\to \ell^{p'}(\Lambda)} \|S_{2,y}\|_{\mathcal{H}_2\to \ell^{\infty}(\Lambda)} \\
    & \times \int_Y \Big(\sum_{x\in\Lambda}|V\one_{\mathcal{V}_x}(\phi_x^{-1}(y))w_x(y)|^{2q}\Big)^{1/2q}\rd\nu(y) \\
    & \lesssim r^{-d/p'} r^{d}(\log\#\Lambda)^{5/2}\|T_1\|_{\mathcal{H}_1\to L^{p'}(X)}\|T_2\|_{\mathcal{H}_2\to L^{\infty}(X)}\Big(\sum_{x\in\Lambda}\|V\|_{L^{\infty}(B(x,2r))}^{2q}\Big)^{1/2q}.
  \end{align*}

(2) For each $x\in\Lambda$, let $\Lambda_r(x)$ be a maximal $r$-separated subset of $\mathcal{V}_x$.
  Let $$\Lambda_r:=\bigcup_{x\in\Lambda}\Lambda_r(x).$$
  By assumption, $(X,\Lambda_r)$ satisfies the uniform reparametrisation property.
  For fixed $y\in Y$, define $S_{i,y}$ as in \eqref{eq: def Siy}
  and $v_{\omega,y}(x):=(V_{\omega}\mathbf{1}_{\mathcal{V}_r(x)})(\phi_x^{-1}(y))w_x(y)$ for $x\in\Lambda_r$, where
     \[
    \mathcal{V}_r(x) := \{y\in X:\rho(x,y)<\rho(x',y)\mbox{ for all }x'\in\Lambda_r\setminus\{x\}\}.
\]
 The remainder of the proof is analogous to part (1); in addition, one uses H\"older and the bound $\#\Lambda_r(x)\lesssim (r'/r)^d$ (see Remark \ref{rem:Holder}).
\end{proof}

\subsection{A sufficient condition for the weak local constancy property}

\section{Compact manifolds}
We start this section with an application of Theorem \ref{thm:abstract2} to a dual form of Sogge's spectral projection bounds. The result will not be used directly to prove Theorem~\ref{thm:main_intro} but is perhaps of independent interest. We then turn to resolvent estimates, which are similar to those in \cite{Cuenin2026}, but for the square root of the resolvent instead of the full resolvent. We then use these estimates, together with Theorem \ref{thm:abstract2}, to prove a Birman--Schwinger bound involving Anderson-type potentials. The proof of Theorem \ref{thm:main_intro} is then a straightforward consequence of this bound.

\subsection{Spectral projections}
Consider the spectral projection operators
\begin{align}
  \Pi_{\lambda} 
  := \mathbf{1}\left(P \in [\lambda,\lambda+1]\right):L^2(M)\to L^{p'}(M)
\end{align}
where $P=\sqrt{-\Delta_g}$. By Sogge's bounds \cite{Sogge1988},
\begin{align}\label{eq: Sogge bound large p'}
  \|\Pi_{\lambda} \|_{L^2(M)\to  L^{p'}(M)}\lesssim \lambda^{\nu(p')}. 
\end{align}
where
\begin{align}\label{def. nu(p')}
  \nu(p') :=
  \begin{cases}
    \frac{d-1}{2}(\frac{1}{2}-\frac{1}{p'})\quad & 2\leq p'\leq \frac{2(d+1)}{d-1},\\
    d(\frac{1}{2}-\frac{1}{p'})-\frac{1}{2}\quad & \frac{2(d+1)}{d-1}\leq  p'\leq \infty.
  \end{cases}   
\end{align}

Here and in what follows, we always assume that $1\leq p\leq 2\leq p'\leq \infty$, 
Given $1\leq q\leq\infty$, we also adopt the convention that $p$ is defined by
\begin{align}\label{eq:1/q=1/p-1/p'}
\frac{1}{q}=\frac{1}{p}-\frac{1}{p'}.
\end{align}
More explicitly, this means that
\[
p=\frac{2q}{q+1},\quad p'=\frac{2q}{q-1}.
\]
We also set
\[
q_c:=\frac{d+1}{2}.
\]
Using this identification in \eqref{def. nu(p')}, we can write 
\begin{align}\label{sigma(q)}
  \nu(p')=\sigma(q):=\begin{cases}
        \frac{d}{2q}-\frac{1}{2}\quad & 1\leq  q\leq q_c,\\
        \frac{d-1}{4q}\quad & q_c\leq q\leq\infty. 
    \end{cases}   
\end{align}

Then, by \eqref{eq: Sogge bound large p'} and H\"older's inequality,
\begin{align}\label{eq: dual Sogge bound large p'}
  \|\Pi_{\lambda} ^*V\Pi_{\lambda} \|_{L^2(M)\to L^2(M)}
  \lesssim \lambda^{2\sigma(q)}\|V\|_{L^q(M)},\quad V\in L^q(M).
\end{align}
Conversely, \eqref{eq: dual Sogge bound large p'} implies \eqref{eq: Sogge bound large p'} by duality.

The following proposition yields an improvement over \eqref{eq: dual Sogge bound large p'} for random potentials.  

\begin{proposition}\label{prop: manifold spectral cluster expectation}
  Let $\lambda\geq 2$ and 
  $$
  r\leq \min\Big(\frac{\mathrm{inj}(M)}{4},\lambda^{-1}\Big)\leq R\leq \diam(M).
  $$
  Assume that $V\in L^1_{\rm \loc}(M)$ and that its support has diameter at most $R$. Let $\Lambda\subset M$ be a maximal $r$-separated set and let $\{\mathcal{V}_x\}_{x\in\Lambda}$ be the associated Voronoi cells. Let
  \begin{align*}
    V_{\omega}(y):=\omega_xV(y)\quad\mbox{for }x\in \Lambda,y\in \mathcal{V}_x,
  \end{align*}
  where $\omega_j$ are i.i.d.~symmetric Bernoulli or centred normalised Gaussian random variables. Then, for $1\leq q\leq\infty$, we have
  \begin{align*}
    \mathbf{E}\|\Pi_{\lambda} ^*V_{\omega}\Pi_{\lambda} \|_{L^2(M)\to L^2(M)}
    & \lesssim \lambda^{-\mu(q)}(\log R/r )^{5/2}\Big(\sum_{x\in\Lambda}\|V\|_{L^{\infty}(B(x,2r))}^{2q}\Big)^{1/(2q)},
  \end{align*}    
  with the obvious modification for $q=\infty$. Here, $\mu(q)$ is given by \eqref{def.mu(q)}
In particular, for any $K\geq 1$, we have
  \begin{align*}
    \|\Pi_{\lambda} ^*V_{\omega}\Pi_{\lambda} \|_{L^2(M)\to L^2(M)}
    & \lesssim K \lambda^{-\mu(q)}(\log R/r )^{5/2}\Big(\sum_{x\in\Lambda}\|V\|_{L^{\infty}(B(x,2r))}^{2q}\Big)^{1/(2q)}
  \end{align*} 
 with probability at least $1-\exp(-K^2)$. 
\end{proposition}

\begin{proof}
  The first statement follows from Theorem \ref{thm:abstract2} (1) together with \eqref{eq: Sogge bound large p'} by taking
  \begin{itemize}
  \item $X=M$ with Riemannian volume measure and the metric space structure induced by the Riemannian metric,
  \item $Y=B(0,2r)\subset \R^d$ with $d$-dimensional Lebesgue measure and $\phi_x:=\exp_x^{-1}$,
  \item $\mathcal{H}_1=\mathcal{H}_2=L^2(M)$ and $T_1=T_2=\Pi_{\lambda}$, 
  \end{itemize}
  and observing that, due to \eqref{eq:1/q=1/p-1/p'}, \eqref{sigma(q)}, 
\[
-d/p+\nu(p')+\nu(\infty)
=
-\mu(q).
\]
  By definition of the injectivity radius, 
  $$
  \exp_x:B(0,\mathrm{inj}(M))\to B(x,\mathrm{inj}(M))
  $$
  is a diffeomorphism for every $x\in M$. Since $4r\leq \mathrm{inj}(M)$, it follows by compactness of $M$ that
  \begin{align*}
    \sup_{x\in M}\|\rd\exp_x^{-1}|_{B(x,2r)}\|<\infty.
  \end{align*}
  Hence, by Lemma \ref{lemma uniform bilipschitz}, Definition \ref{asspt: uniform Radon-Nikodym} is satisfied. The assumption \eqref{eq:abstract comparison ellp Lp} is verified in Lemma~\ref{verificationlocalconstancyspectralcluster2}.
The second statement follows from Lemma \ref{lem:tailbound}.
\end{proof}

\begin{remark}\label{rem:rand.of_fixed_V}
    The class of potentials covered by Proposition \ref{prop: manifold spectral cluster expectation} contains randomisations of any fixed $V$ with finite $\ell^{2q}L^{\infty}$ norm (as in \eqref{eq:defrandomisation}), not just Anderson-type potentials of the form \eqref{eq:Anderson-type}. In the latter case, taking $r=\lambda^{-1}$, $R=\diam(M)$, we get
      \begin{align*}
    \mathbf{E}\|\Pi_{\lambda} ^*V_{\omega}\Pi_{\lambda} \|_{L^2(M)\to L^2(M)}
     \lesssim \lambda^{-\mu(q)}(\log \lambda )^{5/2}\|v(\lambda)\|_{\ell^{2q}}.
  \end{align*}    
In comparison, e.g. in the Rademacher case, the deterministic bound \eqref{eq: dual Sogge bound large p'} yields 
\begin{align*}
\|\Pi_{\lambda} ^*V_{\omega}\Pi_{\lambda} \|_{L^2(M)\to L^2(M)}
     \lesssim \lambda^{2\sigma(2q)-\frac{d}{2q}}\|v(\lambda)\|_{\ell^{2q}}.   
\end{align*}
Since
\begin{align}\label{eq:mu(q)+2sigma(2q)-d/2q}
\mu(q)+2\sigma(2q)-\frac{d}{2q}
=
\begin{cases}
0, & 1\le q\le \dfrac{q_c}{2},\\
1-\frac{q_c}{2q}, & \dfrac{q_c}{2}\le q\le q_c,\\
\frac12, & q_c\le q\leq\infty,
\end{cases}
\end{align}
it follows that $-\mu(q)<2\sigma(2q)-\frac{d}{2q}$ for $q>q_c/2$. Hence, the probabilistic bound improves upon the deterministic one in this range.
\end{remark}

\subsection{Half-resolvent estimates}\label{sect:Half-resolvent estimates}
We prove the following deterministic bounds for a larger range of exponents than needed for the proof of Theorem \ref{thm:main_intro}.
Let 
\begin{align}\label{def:delta(z)}
d(z):=\dist(z,\spec(-\Delta_g)),\quad \langle z\rangle:=2+|z|,\quad \delta(z):=\frac{\min\{d(z),|z|^{1/2}\}}{\log\langle z\rangle}.
\end{align}

\begin{proposition}\label{prop:L2_to_Lp'_resolvent_estimate}
  Let $\Delta_g$ be the Laplace--Beltrami operator on a $d$-dimensional closed Riemannian manifold $(M,g)$. Let $2\leq p'\leq 2d/(d-2)$ if $d\geq 3$ and $2\leq p'<\infty$ if $d=2$.
  Then for all $z\in \C\setminus\spec(-\Delta_g)$,
  \begin{align}
    \||-\Delta_g-z|^{-1/2}\|_{L^{2}(M)\to L^{p'}(M)}
    \lesssim \delta(z)^{-1/2} \langle z\rangle^{\frac{\nu(p')}{2}}.\label{eq: deterministic half resolvent bound}
  \end{align}  
  The implicit constant depends only on $(M,g)$ and $d,p$, but not on $z$.
\end{proposition}

\begin{proof} 
  1. The same argument as in \cite[Lemma 2.3]{Cuenin2026} yields
  \[
    \||-\Delta_g-z|^{-1/2}\|_{L^{2}(M)\to L^{p'}(M)}
    \lesssim d(z)^{-1/2},\quad \re z\leq 1,
  \]
  which is better than \eqref{eq: deterministic half resolvent bound}.

  2. It remains to consider the case $\re z>1$.
  Let $\Sigma:=\{z\in\C:|\im z|\geq \re z\}$. Then for $z\in \Sigma$, we have 
  \[
    \Big\|\frac{(-\Delta_g+|z|)^{1/2}}{|-\Delta_g-z|^{1/2}}\Big\|\lesssim 1.
  \]
  Thus, by Sobolev embedding,
  \begin{align*}
    \||-\Delta_g-z|^{-1/2}f\|_{L^{p'}}
    &=\|(-\Delta_g+|z|)^{-1/2}\frac{(-\Delta_g+|z|)^{1/2}}{|-\Delta_g-z|^{1/2}}f\|_{L^{p'}}\\
    & \lesssim |z|^{\frac{d}{2}(\frac{1}{p}-\frac{1}{2})-\frac{1}{2}}\|\frac{(-\Delta_g+|z|)^{1/2}}{|-\Delta_g-z|^{1/2}}f\|_{L^2}
\lesssim |z|^{\frac{d}{2}(\frac{1}{p}-\frac{1}{2})-\frac{1}{2}}\|f\|_{L^2}
  \end{align*}
for $z\in \Sigma$, which is again better than \eqref{eq: deterministic half resolvent bound} for $\re z>1$. A similar argument shows
  \begin{align*}
    \||-\Delta_g-z|^{-1/2}\mathbf{1}(-\Delta_g\notin[\re z/2,2\re z])\|_{L^{2}(M)\to L^{p'}(M)}
    & \lesssim|z|^{\frac{d}{2}(\frac{1}{p}-\frac{1}{2})-\frac{1}{2}}.
  \end{align*}

  \noindent
  3. It remains to estimate the half resolvent, spectrally localised to $[\re z/2,2\re z]$, for $\re z>\max(1,|\im z|)$. Let $f\in \Ran\mathbf{1}(-\Delta_g\in[\re z/2,2\re z])$. We first observe that
  \begin{align*}
    |-\Delta_g-z|^{-\frac12}f=\sum_{k\in \Z\cap [0,2\sqrt{\re z}]}\Pi_k\Upsilon_{k,z}f,
  \end{align*}
  where 
  $$
  \Upsilon_{k,z}:=\sum_{\lambda_j^2\in [k^2,(k+1)^2)\cap [\re z/2,2\re z] }\Big(\frac{1}{|\lambda_j^2-z|}\Big)^{1/2}E_j,
  $$
  and $E_j$ is the orthogonal projection onto the eigenspace of $-\Delta_g$ corresponding to $\lambda_j^2$.
  Let $k_0:=\ceil{\sqrt{\re z}}$.
  By the spectral theorem, we have the $L^2(M)$-operator norm estimates
  \begin{align}
    \|\Upsilon_{k,z}\|^2\leq 
    \begin{cases}
    \frac{1}{d(z)},\quad &\mbox{for }|k-k_0|\leq 10,\\
    \frac{5}{\sqrt{\re z}|k-k_0|},\quad &\mbox{for }|k-k_0|> 10.
    \end{cases}
  \end{align}
Indeed, for $|k-k_0|> 10$, $k\in \Z\cap [0,2\sqrt{\re z}]$ and $\re z\geq 1$, we have 
\begin{align*}
    |\lambda_j^2-z|
    & \geq |k^2-\re z|-(2k+1)
    \geq \sqrt{\re z}|k-k_0|-2(\sqrt{\re z}+k+1) \\
    & \geq \frac{1}{5}\sqrt{\re z}|k-k_0|.
\end{align*}
  By orthogonality and Sogge's spectral cluster bounds \eqref{eq: Sogge bound large p'}, it follows that 
  \begin{align*}
    \||-\Delta_g-z|^{-\frac12}f\|_{L^2}^2
    &\leq \sum_{k\in \Z\cap [0,2\sqrt{\re z}]} \|\Upsilon_{k,z}\|^2\|\Pi_kf\|_{L^2}^2\\
    &\lesssim (\sqrt{\re z})^{2\nu(p')}\|f\|_{L^p}^2\sum_{k\in \Z\cap [0,2\sqrt{\re z}]} \|\Upsilon_{k,z}\|^2\\
    &\lesssim (\sqrt{\re z})^{2\nu(p')}\Big(\frac{1}{d(z)}+\frac{\log\sqrt{\re z}}{\sqrt{\re z}}\Big)\|f\|_{L^p}^2.
  \end{align*}
 By duality (using that estimates are invariant under $z\to \overline{z}$), we conclude that
  \begin{align*}
    \||-\Delta_g-z|^{-1/2}\mathbf{1}(-\Delta_g\in[\re z/2,2\re z])\|_{L^{2}(M)\to L^{p'}(M)}
    & \lesssim \Big(\frac{1}{d(z)}+\frac{\log |z|}{|z|^{1/2}}\Big)^{1/2} |z|^{\frac{\nu(p')}{2}}
  \end{align*}
  for $\re z>\max(1,|\im z|)$.
\end{proof}

\begin{remark}\label{rem:p'}
We used $p'\leq 2d/(d-2)$ (or $p'<\infty$ if $d=2$) only in Step 2. The estimates in Steps 1 and 3 hold for $2\leq p'\leq \infty$.

\end{remark}

We can also draw the following consequence from the above proof.

\begin{corollary}\label{cor:resolvent_bound_cutoffs}
 Let $P=\sqrt{-\Delta_g}$ and $\lambda\geq 2$. Let $\chi_-$ and $\chi$ be smooth functions on $\R$ supported on the intervals $(-\infty,4)$ and $(1/4,4)$, respectively. Then for all $\lambda\geq 2$, $k\in\N_0$ and all $2\leq p'\leq\infty$,
 \begin{align}
 \||-\Delta_g-z|^{-1/2}\chi_-(P/\lambda)\|_{L^{2}(M)\to L^{p'}(M)}&\lesssim 
 \delta(z)^{-1/2} \langle z\rangle^{\frac{\nu(p')}{2}},\label{eq:resolvent_chi_-}\\
 \||-\Delta_g-z|^{-1/2}\chi(P/2^k)\|_{L^{2}(M)\to L^{p'}(M)}&\lesssim 2^{k(\frac{d}{2}-\frac{d}{p'}-1)}\quad \mbox{for } |z|\leq 2^{2k-5}.\label{eq:resolvent_chi_k}
 \end{align} 
\end{corollary}

\begin{proof}
 By the spectral theorem,
   \[
    \Big\|\frac{\chi(P/2^k)(-\Delta_g+2^{2k})^{1/2}}{|-\Delta_g-z|^{1/2}}\Big\|\leq\sup_{2^{k-2}<\tau<2^{k+2}}\left|\frac{\tau^2+2^{2k}}{\tau^2-z}\right|^{1/2}\lesssim 1\quad \mbox{for } |z|\leq 2^{2k-5}. 
  \]
  Thus, by Bernstein inequalities,
  \begin{align*}
    \||-\Delta_g-z|^{-1/2}\chi(P/2^k)f\|_{L^{p'}}
    & \lesssim 2^{k(\frac{d}{2}-\frac{d}{p'})} \|(-\Delta_g+2^{2k})^{-1/2}f\|_{L^{2}} \\
    & \lesssim 2^{k(\frac{d}{2}-\frac{d}{p'}-1)} \|f\|_{L^2}.
  \end{align*}
The bound \eqref{eq:resolvent_chi_-} follows from \eqref{eq: deterministic half resolvent bound} and Remark \ref{rem:p'}.
\end{proof}

\subsection{Birman--Schwinger operator}
Let us consider the Birman--Schwinger operator
\[
  \mathcal{K}_{\omega,\lambda}(z) := |-\Delta_g-z|^{-\frac12} V_{\omega,\lambda} (-\Delta_g-z)^{-\frac12},\quad z\in\C\setminus\spec(-\Delta_g).
\]
where $V_{\omega,\lambda}$ is of Anderson-type, as in \eqref{eq:Anderson-type}. More precisely, we assume that
\begin{align}\label{eq:Anderson-type-repeated}
  V_{\omega,\lambda}(x)
  = \sum_{j=1}^{N(\lambda)} \omega_j v_j(\lambda) \psi_j(x;\lambda),
  \quad x \in M,
\end{align}
where $\psi_j(\cdot;\lambda)\in C_c^{\infty}(M;[0,1])$, and $E_j(\lambda)=\supp \psi_j(\cdot;\lambda)$ have diameter $\approx \lambda^{-1}$, with bounded overlap
  \[  \sup_{x\in M}\sum_{j=1}^{N(\lambda)}\mathbf{1}_{E_j(\lambda)}(x)\leq C.
  \]  
The coefficients $v_j(\lambda)$ are complex numbers, $\omega_j$ are i.i.d.~symmetric Bernoulli or centred normalised Gaussian random variables, and $\lambda\geq 2$ is a large parameter.

The following is the main result of this subsection.

\begin{proposition}\label{prop:BS_small_q}
  Let $V_{\omega,\lambda}$ be as in \eqref{eq:Anderson-type-repeated}, $\lambda\geq 2$. Then for $z\in\C\setminus\spec(-\Delta_g)$, $|z|\leq \lambda^2$, $\frac{d}{2}\leq q\leq \infty$,
  \begin{align}\label{eq:EnormBSexp}
\mathbf{E}\|\mathcal{K}_{\omega,\lambda}(z)\|
     \lesssim \lambda^{-\frac{d(q+1)}{2q}}\,\delta(z)^{-1}
\,\langle z\rangle^{\frac{d(q+1)}{4q}-\frac{\mu(q)}{2}}
(\log\lambda)^{5/2}\|v(\lambda)\|_{\ell^{2q}},
  \end{align}
where $\delta(z)$ and $\mu(q)$ were defined in \eqref{def:delta(z)} and \eqref{def.mu(q)}, respectively.  
Moreover, for any $K\geq 1$, and $z,q,\lambda$ as above, there exists an event $E_0(z,q,\lambda,K)$ with probability
\begin{align}\label{eq:P(E_0(z,q,lambda,K))}
  \mathbf{P}
  (E_0(z,q,\lambda,K))
  \geq 1-\exp(-K^2)
\end{align}
and such that for all $\omega\in E_0(z,q,\lambda,K)$,
\begin{align}\label{eq:EnormBS}
  \|\mathcal{K}_{\omega,\lambda}(z)\|
  \lesssim K\lambda^{-\frac{d(q+1)}{2q}}\,\delta(z)^{-1}
\,\langle z\rangle^{\frac{d(q+1)}{4q}-\frac{\mu(q)}{2}}
(\log\lambda)^{5/2}\|v(\lambda)\|_{\ell^{2q}}.
\end{align}
\end{proposition}

\begin{remark}
    Similarly to Proposition \ref{prop: manifold spectral cluster expectation}, our method could handle randomisations of a fixed potential, see Remark \ref{rem:rand.of_fixed_V}. For the sake of exposition and compactness of notation, we restrict our attention to Anderson-type potentials.
\end{remark}

In the following, to simplify notation, we use the abbreviations
\begin{align}
A_{\lambda} & := \lambda^{-\frac{d(q+1)}{2q}}
(\log\lambda)^{5/2}\|v(\lambda)\|_{\ell^{2q}},
\label{def:Alambda}
\\
B(z) & := \delta(z)
\,\langle z\rangle^{\frac{\mu(q)}{2}-\frac{d(q+1)}{4q}}.\label{def:B(z)}
\end{align}
The proof of Proposition \ref{prop:BS_small_q} will be given at the end of this subsection.

\begin{lemma}\label{lem:asymmetric_Birman-Schwinger}
 Under the assumptions of Corollary \ref{cor:resolvent_bound_cutoffs}, we have, for $1\leq q\leq \infty$,
 \begin{align}\label{eq:chi_-chi_-}
 \mathbf{E}\|\chi_-(P/\lambda)\mathcal{K}_{\omega,\lambda}(z)\chi_-(P/\lambda)\|&\lesssim A_{\lambda}B(z)^{-1}.
 \end{align}
Moreover, for $|z|\leq 2^{2k-5}$, $\lambda\lesssim 2^k$,
\begin{align}
\mathbf{E}\|\chi(P/2^k)\mathcal{K}_{\omega,\lambda}(z)\chi_-(P/\lambda)\|&\lesssim 2^{-k}\lambda^{-\frac{d}{2}}(\log\lambda)^{5/2}\delta(z)^{-\frac{1}{2}} \langle z\rangle^{\frac{d-1}{4}}\|v(\lambda)\|_{\ell^{2q}},
 \label{eq:chi_k_chi_-}
 \end{align}
 and for $|z|\leq 2^{2\ell-5}$ and $k\geq \ell$,
 \begin{align}
\mathbf{E}\|\chi(P/2^\ell)\mathcal{K}_{\omega,\lambda}(z)\chi(P/2^k)\|&\lesssim 2^{-k}\lambda^{-\frac{d}{2}}(\log\lambda)^{5/2}2^{\ell(\frac{d}{2}-1)}\|v(\lambda)\|_{\ell^{2q}}.
 \label{eq:chi_l_chi_k}
 \end{align}
\end{lemma}

\begin{proof}
The proof uses Theorem~\ref{thm:abstract2}, whose assumptions we verified in Proposition~\ref{prop: manifold spectral cluster expectation}. 

\smallskip
1. Theorem \ref{thm:abstract2} (1), with $r=\lambda^{-1}$, together with \eqref{eq:resolvent_chi_-} yields
\begin{align*}
 \mathbf{E}\|\chi_-(P/\lambda)\mathcal{K}_{\omega,\lambda}(z)\chi_-(P/\lambda)\|
 &\lesssim 
 \lambda^{-d/p} (\log\lambda)^{5/2} \delta(z)^{-1} \langle z\rangle^{\frac{\nu(p')+\nu(\infty)}{2}}\|v(\lambda)\|_{\ell^{2q}}.
\end{align*}
Observing that $\nu(p')+\nu(\infty)=d/p-\mu(q)$ and recalling \eqref{eq:1/q=1/p-1/p'}, \eqref{sigma(q)}, this proves \eqref{eq:chi_-chi_-}.

2. Theorem \ref{thm:abstract2} (2), with $r=2^{-k}$, $r'=\lambda^{-1}$, together with \eqref{eq:resolvent_chi_-} and \eqref{eq:resolvent_chi_k} implies
\begin{align*}
 \mathbf{E}\|\chi_-(P/\lambda)\mathcal{K}_{\omega,\lambda}(z)\chi(P/2^k)\|
 &\lesssim 
 2^{-\frac{kd}{p}} (2^k/\lambda)^{\frac{d}{p}-\frac{d}{2q}}(\log\lambda)^{5/2}\\
&\quad\times 2^{k(\frac{d}{2}-\frac{d}{p'}-1)}\delta(z)^{-1/2} \langle z\rangle^{\frac{\nu(\infty)}{2}}\|v(\lambda)\|_{\ell^{2q}}.
\end{align*}
Since $\nu(\infty)=\frac{d-1}{2}$,
\(
2^{k(-\frac{d}{2q}+\frac{d}{2}-\frac{d}{p'}-1)}=2^{-k}
\)
and $\lambda^{\frac{d}{2q}-\frac{d}{p}}=\lambda^{-\frac{d}{2}}$,
this proves \eqref{eq:chi_k_chi_-}.

3. Similarly, Theorem \ref{thm:abstract2} (2) and \eqref{eq:resolvent_chi_k} yield
\begin{align*}
 \mathbf{E}\|\chi(P/2^\ell)\mathcal{K}_{\omega,\lambda}(z)\chi(P/2^k)\|
 &\lesssim 
 2^{-\frac{kd}{p}} (2^k/\lambda)^{\frac{d}{p}-\frac{d}{2q}}(\log\lambda)^{5/2}\\
&\quad\times 2^{k(\frac{d}{2}-\frac{d}{p'}-1)}2^{\ell(\frac{d}{2}-1)}\|v(\lambda)\|_{\ell^{2q}},
\end{align*}
which proves \eqref{eq:chi_l_chi_k}.
\end{proof}

\begin{corollary}\label{cor:chi_+chi_-and_chi_+chi_+}
Let $\lambda\geq 2$ and $\chi_-,\chi_+\in C^{\infty}(\R;[0,1])$ such that $\supp\chi_-\subset (-\infty,8)$ and $\supp\chi_+\subset (4,\infty)$. Then for $z\in\C\setminus\spec(-\Delta_g)$, $|z|\leq \lambda^2$, $1\leq q\leq q_c$,
\begin{align}
\mathbf{E}\|\chi_+(P/\lambda)\mathcal{K}_{\omega,\lambda}(z)\chi_-(P/\lambda)\|&\lesssim \lambda^{-\frac{d}{2}-1}(\log\lambda)^{5/2}\delta(z)^{-\frac{1}{2}} \langle z\rangle^{\frac{d-1}{4}}\|v(\lambda)\|_{\ell^{2q}},\label{eq:chi_+_chi_-}\\
\mathbf{E}\|\chi_+(P/\lambda)\mathcal{K}_{\omega,\lambda}(z)\chi_+(P/\lambda)\|&\lesssim \lambda^{-2}(\log\lambda)^{5/2}\|v(\lambda)\|_{\ell^{2q}}\label{eq:chi_+_chi_+}.
\end{align}
\end{corollary}

\begin{proof}
Let $\chi\in C_c^{\infty}(\R)$ be such that $\supp\chi\subset (1/4,4)$ and $\chi=1$ on $(1/2,2)$.
Then 
\begin{align*}
\mathbf{E}\|\chi_+(P/\lambda)\mathcal{K}_{\omega,\lambda}(z)\chi_-(P/\lambda)\|\leq \sum_{2^{k}\geq 8\lambda}\mathbf{E}\|\chi(P/2^k)\mathcal{K}_{\omega,\lambda}(z)\chi_-(P/\lambda)\|,
\end{align*}    
and \eqref{eq:chi_+_chi_-} follows from  \eqref{eq:chi_k_chi_-} since $|z|\leq \lambda^2\leq 2^{2k-6}$ in the range of summation.
Similarly, 
\begin{align*}
\mathbf{E}\|\chi_+(P/\lambda)\mathcal{K}_{\omega,\lambda}(z)\chi_+(P/\lambda)\|\leq \sum_{2^{k}\geq 8\lambda}\sum_{2^{\ell}\geq 8\lambda}\mathbf{E}\|\chi(P/2^\ell)\mathcal{K}_{\omega,\lambda}(z)\chi(P/2^k)\|,    
\end{align*}
and \eqref{eq:chi_l_chi_k} implies that the part of the sum with $\ell\leq k$ is bounded by the right-hand side of \eqref{eq:chi_+_chi_+}. Replacing $z$ by $\overline{z}$ in \eqref{eq:chi_+_chi_-} and using the bound for the adjoint, the same argument applies to the part of the sum with $\ell>k$.
\end{proof}

\begin{proof}[Proof of Proposition \ref{prop:BS_small_q}]
Let $\chi_-,\chi_+$ be such that $\chi_-+\chi_+=1$ and with support as in Corollary \ref{cor:chi_+chi_-and_chi_+chi_+}. Then by \eqref{eq:chi_-chi_-}, \eqref{eq:chi_+_chi_-} and its dual (with $z$ replaced by $\overline{z}$), and \eqref{eq:chi_+_chi_+},
\begin{align*}
\mathbf{E}\|\mathcal{K}_{\omega,\lambda}(z)\|
&\leq \sum_{\sigma_1,\sigma_2\in\{+,-\}}\mathbf{E}\|\chi_{\sigma_1}(P/\lambda)\mathcal{K}_{\omega,\lambda}(z)\chi_{\sigma_2}(P/\lambda)\|
\\
 &\lesssim
 (\lambda^{-\frac{d}{p}}  \delta(z)^{-1} \langle z\rangle^{\frac{d}{2p}-\frac{\mu(q)}{2}}
 +\lambda^{-\frac{d}{2}-1}\delta(z)^{-\frac{1}{2}} \langle z\rangle^{\frac{d-1}{4}}+\lambda^{-2})
 (\log\lambda)^{5/2}\|v(\lambda)\|_{\ell^{2q}}.
\end{align*}
Using $\delta(z)\leq \langle z\rangle^{\frac{1}{2}}\lesssim\lambda$ and $q\geq \frac{d}{2}$, we may bound the second and third term in the last expression by the first.
This proves \eqref{eq:EnormBSexp}, which 
together with Lemma~\ref{lem:tailbound} yields~\eqref{eq:EnormBS}.
\end{proof}

\begin{remark}
 The above proof is the only place where we used the assumption $q\geq d/2$.   
\end{remark}

\subsection{Discretisation and union bound}
We cannot directly apply Proposition \ref{prop:BS_small_q} to prove Theorem \ref{thm:main_intro}.
The problem is that the events on which the corresponding inequalities hold depend on the spectral parameter \(z\). 
Since \(z\) ranges over a continuum, one cannot directly apply a union bound over all \(z\). To overcome this, we will introduce a discretisation scheme.


To carry out the discretisation, let
\begin{align}
    E(z,q,\lambda,K)&:=\{\omega:\|\mathcal{K}_{\omega,\lambda}(z)\|\lesssim K A_{\lambda}B(z)^{-1}\},\label{def:E(z)}
\end{align}
where $A_{\lambda}$, $B(z)$ were defined in \eqref{def:Alambda}, \eqref{def:B(z)}, respectively.
In what follows, we will abbreviate \(E(z,q,\lambda,K)\) by \(E(z)\).
The result of Proposition \ref{prop:BS_small_q} can then be stated as
\begin{align}\label{eq:P(E(z))}
  \mathbf{P}(E(z))\geq 1-\exp(-K^2).
\end{align}
The aim of this section is to prove the following proposition, which is a uniform version of the above bound. For technical reasons, we need to restrict $z$ to the complement of an arbitrarily small neighborhood of $\spec(-\Delta_g)$.

\begin{proposition}\label{prop:uniform-on-truncated-region}
Let \(N\geq 1\) be fixed. Then there exists a constant \(C_N>0\) such that
for every \(K\geq 1\),
\begin{align}\label{eq:uniform-on-truncated-region}
\mathbf P\Bigl(
\sup_{\substack{|z|\leq \lambda^2\\ d(z)\geq \lambda^{-N}}}
B(z)\,\|\mathcal K_{\omega,\lambda}(z)\|
\leq
C_NK\,A_\lambda (\log\lambda)^2
\Bigr)
\geq
1-\exp(-K^2).
\end{align}
\end{proposition}

As a preliminary step, we record the following simple Lipschitz bound for the Birman-Schwinger operator.

\begin{lemma}\label{lem:BS-perturb-z}
  For any \(z,z'\in \C\setminus \spec(-\Delta_g)\) with $d(z),d(z')\geq  \rho>0$, we have
  \[
    \|\mathcal{K}_{\omega,\lambda}(z)-\mathcal{K}_{\omega,\lambda}(z')\|
    \lesssim |z-z'| \, \rho^{-2}\, \|V_{\omega,\lambda}\|_{L^{\infty}(M)}.
  \]
\end{lemma}

\begin{proof}
  Set
\[
R(w):=(-\Delta_g-w)^{-1/2},
\qquad
w\in \C\setminus \spec(-\Delta_g).
\]
Then
\[
\mathcal{K}_{\omega,\lambda}(w)=|R(w)|V_{\omega,\lambda}R(w),
\]
and hence
\[
\mathcal{K}_{\omega,\lambda}(z)-\mathcal{K}_{\omega,\lambda}(z')
=
\bigl(|R(z)|-|R(z')|\bigr)V_{\omega,\lambda}R(z)
+
|R(z')|V_{\omega,\lambda}\bigl(R(z)-R(z')\bigr).
\]
Therefore, taking the operator norm,
\begin{align*}
\|\mathcal{K}_{\omega,\lambda}(z)-\mathcal{K}_{\omega,\lambda}(z')\|
&\leq
\|R(z)-R(z')\|\,\|V_{\omega,\lambda}\|\,\|R(z)\| \\
&\qquad
+\|R(z')\|\,\|V_{\omega,\lambda}\|\,\|R(z)-R(z')\|,
\end{align*}
where $\|V_{\omega,\lambda}\|=\|V_{\omega,\lambda}\|_{L^{\infty}(M)}$.
It remains to estimate \(\|R(w)\|\) and \(\|R(z)-R(z')\|\).

By the spectral theorem,
\begin{align}\label{eq:trivial bound R(w)}
\|R(w)\|
=
\sup_{\mu\in \spec(-\Delta_g)} |\mu-w|^{-1/2}
=
d(w)^{-1/2}.
\end{align}

For the difference, define for \(\mu\in \spec(-\Delta_g)\)
\[
f_w(\mu):=|\mu-w|^{-1/2}.
\]
Then
\[
\|R(z)-R(z')\|
=
\sup_{\mu\in \spec(-\Delta_g)} |f_z(\mu)-f_{z'}(\mu)|.
\]
Using 
\begin{align*}
\bigl||a|^{-1/2}-|b|^{-1/2}\bigr|
&=
\frac{\bigl||b|^{1/2}-|a|^{1/2}\bigr|}{|a|^{1/2}|b|^{1/2}}
=
\frac{|\,|b|-|a|\,|}{|a|^{1/2}|b|^{1/2}(|a|^{1/2}+|b|^{1/2})}\\
&\leq
\frac{|a-b|}{|a|^{1/2}|b|^{1/2}(|a|^{1/2}+|b|^{1/2})},
\end{align*}
with \(a=\mu-z\) and \(b=\mu-z'\), we obtain
\[
|f_z(\mu)-f_{z'}(\mu)|
\leq
\frac{|z-z'|}{|\mu-z|^{1/2}|\mu-z'|^{1/2}
\bigl(|\mu-z|^{1/2}+|\mu-z'|^{1/2}\bigr)}.
\]
Since
\[
\frac{1}{x^{1/2}y^{1/2}(x^{1/2}+y^{1/2})}
\leq
\min(x^{-1/2}y^{-1},x^{-1}y^{-1/2})
\leq
x^{-3/2}+y^{-3/2},
\]
for all \(x,y>0\), it follows that
\[
|f_z(\mu)-f_{z'}(\mu)|
\lesssim
|z-z'|\bigl(|\mu-z|^{-3/2}+|\mu-z'|^{-3/2}\bigr).
\]
Taking the supremum over \(\mu\in \spec(-\Delta_g)\) yields
\[
\|R(z)-R(z')\|
\lesssim
|z-z'|\bigl(d(z)^{-3/2}+d(z')^{-3/2}\bigr).
\]

Substituting these bounds into the previous estimate gives
\begin{align*}
\|\mathcal{K}_{\omega,\lambda}(z)-\mathcal{K}_{\omega,\lambda}(z')\|
&\lesssim
|z-z'|\,\|V_{\omega,\lambda}\|
\bigl(d(z)^{-3/2}+d(z')^{-3/2}\bigr)
\bigl(d(z)^{-1/2}+d(z')^{-1/2}\bigr).
\end{align*}
If \(d(z),d(z')\geq \rho\), then
\[
\|\mathcal{K}_{\omega,\lambda}(z)-\mathcal{K}_{\omega,\lambda}(z')\|
\lesssim
|z-z'|\,\rho^{-2}\,\|V_{\omega,\lambda}\|_{L^{\infty}(M)}.
\]
This completes the proof.
\end{proof}

We decompose the set \(\{z\in \C\setminus \spec(-\Delta_g):|z|\leq \lambda^2\}\)
according to the size of \(B(z)\). More precisely, for $j\in\Z$, define
\begin{align*}
\Omega_j&:=\Bigl\{z\in \C\setminus\spec(-\Delta_g): |z|\leq \lambda^2,\ 
2^{j}\leq B(z)<2^{j+1}\Bigr\}.
\end{align*}
Then,
$$
\{z\in \C\setminus \spec(-\Delta_g):|z|\leq \lambda^2\} = \bigcup_{j\in\Z}\Omega_j
$$
and on each \(\Omega_j\), by definition\eqref{def:E(z)},
\begin{align}\label{eq:EnormBS_j}
    \|\mathcal{K}_{\omega,\lambda}(z)\|
\lesssim
K\,2^{-j}A_\lambda,
\qquad z\in \Omega_j,\quad \omega\in E(z).
\end{align}

\begin{corollary}\label{cor:tail_bound_K(z)-K(z')}
Let \(j\in \mathbb Z\) be such that \(\Omega_j\neq \emptyset\). Then there exist $C,c>0$ such that for any $K\geq 1$, there is an event $E_j'$ of probability at least $1-\exp(-K^2)$ such that for all $z\in\Omega_j$, $|z-z'|\leq c2^{j}$ and $\omega\in E_j'$, we have
  \begin{align}\label{eq:tail_bound_K(z)-K(z')}
    \|\mathcal{K}_{\omega,\lambda}(z)-\mathcal{K}_{\omega,\lambda}(z')\|\leq CK2^{-j}\sqrt{\log\lambda} \|v(\lambda)\|_{\ell^{\infty}}.
  \end{align}
\end{corollary}

\begin{proof}
By Lemma \ref{lem:BS-perturb-z}, the map
\[
\Omega_j\ni z\mapsto \mathcal{K}_{\omega,\lambda}(z)
\]
is stable under perturbations of size \(c2^{j}\) for $c$ sufficiently small. 
More precisely, since 
\begin{align*}
d(z)\geq  B(z),\quad
d(z')\geq d(z)-|z-z'|,
\end{align*}
Lemma~\ref{lem:BS-perturb-z} with $\rho=2^j$ yields
\begin{align}\label{eq:K(z)-K(z')_for_z_in_Omega_j}
   z\in\Omega_j,\quad |z-z'|\leq c2^{j}\implies \|\mathcal{K}_{\omega,\lambda}(z)-\mathcal{K}_{\omega,\lambda}(z')\|\lesssim 2^{-j}\|V_{\omega,\lambda}\|_{L^{\infty}(M)}.
\end{align}
Hence, the claimed bound \eqref{eq:tail_bound_K(z)-K(z')} follows immediately from
\eqref{eq:K(z)-K(z')_for_z_in_Omega_j} and the $q=\infty$ case in Lemma \ref{lem:Linfty_norm} below. 
\end{proof}

\begin{lemma}\label{lem:Linfty_norm}
Let $1\leq q <\infty$.
There exists $C_q>0$ such that for any $K\geq 1$, 
\begin{align}
\mathbf{P}(\|V_{\omega,\lambda}\|_{L^{q}(M)}\leq C_qK\lambda^{-\frac{d}{q}}\|v(\lambda)\|_{\ell^{q}})
\geq 1-\exp(-K^2).
\end{align}
Moreover, for $q=\infty$, there exists $C$ such that
\begin{align}
\mathbf{P}(\|V_{\omega,\lambda}\|_{L^{\infty}(M)}\leq CK\sqrt{\log\lambda}\|v(\lambda)\|_{\ell^{\infty}})
\geq 1-\exp(-K^2).
\end{align}
\end{lemma}

\begin{proof}
We first consider the case $q<\infty$.
By the bounded overlap assumption,
\[|V_{\omega,\lambda}(x)|^q
\lesssim_q
\sum_{j=1}^{N(\lambda)}
|\omega_j|^q |v_j(\lambda)|^q \mathbf{1}_{E_j(\lambda)}(x).\]
Integrating over \(M\) and using that \(|E_j(\lambda)|\lesssim \lambda^{-d}\),
\[\|V_{\omega,\lambda}\|_{L^q(M)}^q
\lesssim_q
\lambda^{-d}
\sum_{j=1}^{N(\lambda)}
|\omega_j|^q |v_j(\lambda)|^q.\]
Taking expectations and using \(\mathbf{E}|\omega_j|^q \lesssim_q 1\) for Bernoulli or Gaussian \(\omega_j\),
\[
\mathbf{E}\|V_{\omega,\lambda}\|_{L^q(M)}^q
\lesssim_q
\lambda^{-d}\sum_{j=1}^{N(\lambda)}|v_j(\lambda)|^q
=
\lambda^{-d}\|v(\lambda)\|_{\ell^q}^q.
\]
By Jensen's inequality,
\[
\mathbf{E}\|V_{\omega,\lambda}\|_{L^q(M)}\leq (\mathbf{E}\|V_{\omega,\lambda}\|_{L^q(M)}^q)^{1/q}\lesssim_q \lambda^{-d/q}\|v(\lambda)\|_{\ell^q}.
\]
Lemma \ref{lem:tailbound} then implies the claim.

Now assume $q=\infty$.
By Dudley's inequality \eqref{eq:dudley},
\begin{align*}
\mathbf{E}\|V_{\omega,\lambda}\|_{L^{q}(M)}
&\lesssim \mathbf{E}\max_{j\leq N(\lambda)}|\omega_jv_j(\lambda)|\lesssim \sqrt{\log N(\lambda)} \max_{j\leq N(\lambda)}\|\omega_jv_j(\lambda)\|_{\psi_2}.
\end{align*}
Since $N(\lambda)\lesssim\lambda^d$, it follows that
\begin{align}
\mathbf{E}\|V_{\omega,\lambda}\|_{L^{\infty}(M)}
\lesssim\sqrt{\log\lambda} \|v(\lambda)\|_{\ell^{\infty}}.
\end{align}
Again, Lemma \ref{lem:tailbound} implies the claim.
\end{proof}

\begin{remark}
In the Bernoulli case, the bounds of Lemma \ref{lem:Linfty_norm} are deterministic:
for $1\leq q<\infty$,
\[
\|V_{\omega,\lambda}\|_{L^{q}(M)}
\leq C_q \lambda^{-d/q}\|v(\lambda)\|_{\ell^{q}},
\]
while for $q=\infty$,
\[
\|V_{\omega,\lambda}\|_{L^{\infty}(M)}
\leq C\|v(\lambda)\|_{\ell^{\infty}}.
\]
Hence, only the Gaussian case requires a probabilistic argument.
\end{remark}


Continuing with the discretisation argument, for fixed $N\geq 1$, we set
\[
\Omega_{j,N}:=\Omega_j\cap \{z\in\C:d(z)\geq \lambda^{-N}\}.
\]

\begin{lemma}\label{prop:uniform-on-Omegaj}
Let \(j\in \mathbb Z\) be such that \(\Omega_{j,N}\neq \emptyset\). There exists $C_N>0$ such that for any $K\geq 1$,
\begin{align}\label{eq:uniform-on-Omegaj}
\mathbf P\Bigl(
\sup_{z\in \Omega_{j,N}}\|\mathcal K_{\omega,\lambda}(z)\|
\leq
C_NK\,2^{-j}A_\lambda \log \lambda
\Bigr)
\geq
1-\exp(-K^2).
\end{align}
\end{lemma}

\begin{proof}
Fix \(j\in \mathbb Z\) with \(\Omega_{j,N}\neq \emptyset\), and let
\(\mathcal N_j\subset \Omega_{j,N}\) be a \(c2^j\)-net. We denote the
elements of \(\mathcal N_j\) by \(z_\alpha\). For each \(z_\alpha\in \mathcal N_j\), let
\[
E(z_\alpha)
:=
\Bigl\{
\omega:
\|\mathcal K_{\omega,\lambda}(z_\alpha)\|
\leq
CK\,2^{-j}A_\lambda
\Bigr\}.
\]
By \eqref{eq:EnormBS_j}, we have
\[
\mathbf P(E(z_\alpha))\geq 1-\exp(-K^2).
\]
Next, let \(E'_j\) denote the event on which the perturbation estimate Corollary \ref{cor:tail_bound_K(z)-K(z')} holds simultaneously for all pairs \(z,z_\alpha\) with
\(z\in \Omega_{j,N}\), \(z_\alpha\in\mathcal N_j\), and
\(|z-z_\alpha|\leq c2^j\), namely
\[
\|\mathcal K_{\omega,\lambda}(z)-\mathcal K_{\omega,\lambda}(z_\alpha)\|
\leq
CK\,2^{-j}\sqrt{\log\lambda}\,\|v(\lambda)\|_{\ell^\infty}.
\]
Since
\[
\sqrt{\log\lambda}\,\|v(\lambda)\|_{\ell^\infty}
\leq
(\log\lambda)^{5/2}\|v(\lambda)\|_{\ell^{2q}}
=
A_\lambda,
\]
it follows that on \(E'_j\),
\[
\|\mathcal K_{\omega,\lambda}(z)-\mathcal K_{\omega,\lambda}(z_\alpha)\|
\leq
CK\,2^{-j}A_\lambda
\]
for every \(z\in \Omega_{j,N}\) and every net point \(z_\alpha\) satisfying
\(|z-z_\alpha|\leq c2^j\).

Hence, on the event
\[
F_j:=\Bigl(\bigcap_{z_\alpha\in\mathcal N_j}E(z_\alpha)\Bigr)\cap E'_j,
\]
we have, for every \(z\in \Omega_{j,N}\), choosing \(z_\alpha\in \mathcal N_j\) with
\(|z-z_\alpha|\leq c2^j\),
\[
\|\mathcal K_{\omega,\lambda}(z)\|
\leq
\|\mathcal K_{\omega,\lambda}(z_\alpha)\|
+
\|\mathcal K_{\omega,\lambda}(z)-\mathcal K_{\omega,\lambda}(z_\alpha)\|
\leq
CK\,2^{-j}A_\lambda.
\]
Thus,
\[
F_j\subseteq
\Bigl\{
\omega:
\sup_{z\in \Omega_{j,N}}\|\mathcal K_{\omega,\lambda}(z)\|
\leq
CK\,2^{-j}A_\lambda
\Bigr\}.
\]

It remains to estimate \(\mathbf P(F_j)\). By the union bound,
\[
\mathbf P(F_j^c)
\leq
\sum_{z_\alpha\in \mathcal N_j}\mathbf P(E(z_\alpha)^c)
+\mathbf P((E_j')^c)
\leq
\#\mathcal N_j\exp(-K^2)+\mathbf P(E_j'^c).
\]
By a volume comparison argument, the cardinality of \(\mathcal N_j\) satisfies 
\[
\#\mathcal N_j\leq \frac{\lambda^4}{2^{2j}}
\]
and by Corollary \ref{cor:tail_bound_K(z)-K(z')}, $\mathbf P(E_j'^c)\leq \exp(-K^2)$.
Hence
\[
\mathbf P(F_j^c)
\leq
(1+C\,2^{-2j}\lambda^4) \exp(-K^2).
\]
Since \(\Omega_j\neq \emptyset\) implies 
\begin{align}\label{eq:2j_poly_bounded}
\lambda^{-N-1}\lesssim 2^j\lesssim \lambda^3  
\end{align}
the factor \(2^{-2j}\lambda^4\) is bounded by a fixed power of
\(\lambda\). Replacing \(K\) by \(C_NK\log\lambda\) in the preceding argument, we
obtain the claimed bound \eqref{eq:uniform-on-Omegaj}.
\end{proof}

\begin{proof}[Proof of Proposition \ref{prop:uniform-on-truncated-region}]
For each \(j\in \mathbb Z\), let
\[
F_j:=
\Bigl\{
\omega:
\sup_{z\in \Omega_{j,N}}\|\mathcal K_{\omega,\lambda}(z)\|
\leq
C_NK\,2^{-j}A_\lambda \log\lambda
\Bigr\}.
\]
By Lemma \ref{prop:uniform-on-Omegaj},
\[
\mathbf P(F_j)\geq 1-\exp(-K^2)
\]
for every \(j\) such that \(\Omega_{j,N}\neq \emptyset\).

We first estimate the number of such \(j\). If \(\Omega_{j,N}\neq \emptyset\), then
there exists \(z\in \Omega_{j,N}\) such that
\begin{align}\label{eq:d(z)_approx_2j}
2^j\leq B(z)<2^{j+1}.
\end{align}
In view of \eqref{eq:2j_poly_bounded}, we have
\[
-N\log_2\lambda-O(1)\leq j\leq 3\log_2\lambda+O(1),
\]
so the number of indices \(j\) for which \(\Omega_{j,N}\neq \emptyset\) is
\(O_N(\log\lambda)\).

Now define
\[
F:=\bigcap_{\Omega_{j,N}\neq \emptyset} F_j.
\]
By the union bound,
\[
\mathbf P(F^c)
\leq
\sum_{\Omega_{j,N}\neq \emptyset}\mathbf P(F_j^c)
\leq
C_N(\log\lambda)\exp(-K^2).
\]
Absorbing the factor \(\log\lambda\) by replacing \(K\) with \(C_NK\log\lambda\), we obtain
\[
\mathbf P(F)\geq 1-\exp(-K^2).
\]

It remains to identify the bound on \(F\). Let \(z\) satisfy \(|z|\leq \lambda^2\)
and \(d(z)\geq \lambda^{-N}\). Then \(z\in \Omega_{j,N}\) for some \(j\), and on
\(\Omega_{j,N}\) we have \eqref{eq:d(z)_approx_2j}.
Thus, on the event \(F\),
\[
\|\mathcal K_{\omega,\lambda}(z)\|
\leq
C_NK\,2^{-j}A_\lambda (\log\lambda)^2
\lesssim
C_NK\,B(z)^{-1}A_\lambda (\log\lambda)^2.
\]
Taking the supremum over all such \(z\)
after multiplying by \(B(z)\) yields
\eqref{eq:uniform-on-truncated-region}.
\end{proof}

\subsection{Proof of Theorem \ref{thm:main_intro}}
\begin{proof}
By the Birman--Schwinger principle, $z \in \mathbb{C}\setminus \spec(-\Delta_g)$ is an eigenvalue of $H_{\omega,\lambda}$ if and only if $-1$ is an eigenvalue of $\mathcal{K}_{\omega,\lambda}(z)$. In particular,
\begin{equation}\label{eq:BSnorm}
z \in \spec(H_{\omega,\lambda}) \implies \|\mathcal{K}_{\omega,\lambda}(z)\|\ge 1.
\end{equation}

By Proposition \ref{prop:uniform-on-truncated-region}, there exists an event of probability at least $1 - \exp(-K^2)$ on which the estimate
\begin{equation}\label{eq:BSbound}
\|\mathcal{K}_{\omega,\lambda}(z)\| \leq C_N K A_{\lambda}(\log\lambda)^2B(z)^{-1}
\end{equation}
holds uniformly for all $z\in\C\setminus\spec(-\Delta_g)$ with $|z|\le \lambda^2$ and $d(z)\geq \lambda^{-N}$. Combining \eqref{eq:BSnorm} and \eqref{eq:BSbound} and recalling the definitions of $A_{\lambda}$, $B(z)$, $\delta(z)$ in \eqref{def:Alambda}, \eqref{def:B(z)}, \eqref{def:delta(z)}, respectively, we obtain that for any eigenvalue $z \in \mathbb{C}\setminus \spec(-\Delta_g)$ of $H_{\omega,\lambda}$ with $|z|\le \lambda^2$ and $d(z)\geq \lambda^{-N}$,
\begin{align}\label{eq:keybound}
\min\{d(z), |z|^{1/2}\}
\leq C_N K \langle z\rangle^{\frac{d(q+1)}{4q}-\frac{\mu(q)}{2}}  \, \lambda^{-\frac{d(q+1)}{2q}} (\log \lambda)^{11/2} \|v(\lambda)\|_{\ell^{2q}}. 
\end{align}
Here we have estimated $\log\langle z\rangle \lesssim \log \lambda$.

We now distinguish two cases.

\medskip
\noindent
\textbf{Case 1:} $d(z) \le |z|^{1/2}$.

Let $\lambda_k^2 \in \spec(-\Delta_g)$ be such that $d(z) = |z - \lambda_k^2|$. Then 
\begin{align}
|z|\approx (1+\lambda_k)^2.
\end{align}
Inserting this into \eqref{eq:keybound}, we obtain
\begin{align}\label{eq:d(z)leq case 1}
d(z)
&\leq C_N 
K (1+\lambda_k)^{\frac{d(q+1)}{2q}-\mu(q)}  \, \lambda^{-\frac{d(q+1)}{2q}} (\log \lambda)^{11/2} \|v(\lambda)\|_{\ell^{2q}},
\end{align}
which means that
\[
z\in \bigcup_{\lambda_k\leq \lambda} D(\lambda_k^2,C_NKr_k(\lambda,q)).
\]

\medskip
\noindent
\textbf{Case 2:} $d(z) \ge  |z|^{1/2}$. 

In this case, \eqref{eq:keybound} yields 
\[
|z|^{1/2}\langle z\rangle^{-\frac{d(q+1)}{4q}+\frac{\mu(q)}{2}}  \leq C_NK\, \lambda^{-\frac{d(q+1)}{2q}} (\log \lambda)^{11/2} \|v(\lambda)\|_{\ell^{2q}},
\]
hence $z\in \Omega$.
This proves the claimed spectral inclusion.
\end{proof}

\appendix

\section{Local constancy on closed Riemannian manifolds}
\label{s:localconstancy}

In this appendix, we verify the local constancy property for closed Riemannian manifolds that we used in the proofs of Proposition~\ref{prop: manifold spectral cluster expectation} and Lemma~\ref{lem:asymmetric_Birman-Schwinger}.

\begin{lemma}\label{verificationlocalconstancyspectralcluster1}
 Let $P=\sqrt{-\Delta_g}$, $0<\epsilon<{\rm inj}(M)/10$, and $\chi\in\mathcal{S}(\R)$ be such that $\supp(\widehat{\chi})\subset [-\epsilon,\epsilon]$ and $\chi\geq \mathbf{1}_{[-1,1]}$. Then for all $x\in M$,
 \begin{align*}
  |\chi(P/\lambda)f(x)|\lesssim \lambda^d \int_M (1+\lambda d_g(x,y))^{-N} |f(y)|\rd y+\int_M|f(y)|\rd y.
 \end{align*}
\end{lemma}

\begin{remark}
  This shows that $\Ran \chi(P/\lambda)$ has the local constancy property for any scale $h\leq \lambda^{-1}$.
\end{remark}

\begin{proof}
  We use the Fourier inversion formula to write
  \begin{align*}
    \chi(P/\lambda)=\frac{\lambda}{2\pi}\int_{-\infty}^{\infty}\widehat{\chi}(\lambda t)\e^{\I tP}\rd t.
  \end{align*}
  By \cite[Theorem~4.1.2]{Sogge2017}, for $|t|<\epsilon$,
  \begin{align}\label{eq: Q(t)+R(t)}
    \e^{\I tP}=Q(t)+R(t),
  \end{align}
  where $R(t)$ has a smooth kernel $R(t,x,y)$, the kernel $Q(t,x,y)$ of $Q(t)$ is supported in a small neighborhood of the diagonal in $M\times M$, and in local coordinates takes the form
  \[
    Q(t,x,y) = (2\pi)^{-d}\int_{\R^d}\e^{\I[\phi(x,y,\xi)+tp(y,\xi)]}q(t,x,y,\xi)\rd \xi.
  \]
  Here, $p(x,\xi)=|\xi|_{g(x)}$ is the principal symbol of $P$, $q\in S^0$ is a symbol of order zero, and $\phi$ is homogeneous of degree one in $\xi$ and satisfies
  \begin{align}\label{Taylor expansion of phase function}
    \phi(x,y,\xi)=(x-y)\cdot\xi+\mathcal{O}(|x-y|^2|\xi|).
  \end{align}
  Thus, for $f$ supported in a relatively compact subset of a given coordinate patch,
  \begin{align*}
    \chi(P/\lambda)f(x)&=\frac{\lambda}{(2\pi)^{d+1}}\int_{\R^d}\int_{\R^d}\int_{-\infty}^{\infty}\widehat{\chi}(\lambda t)\e^{\I[\phi(x,y,\xi)+tp(y,\xi)]}q(t,x,y,\xi)f(y)\rd t\rd \xi \rd y \\
    & + \frac{\lambda}{2\pi}\int_{\R^d}\int_{-\infty}^{\infty}\widehat{\chi}(\lambda t)R(t,x,y)f(y)\rd t \rd y.
  \end{align*}
  The modulus of the second term is bounded by a multiple of $\|f\|_{L^1(M)}$. 
  By the change of variables $(t,\xi)\mapsto(\lambda^{-1}t,\lambda\xi)$, the kernel of the operator appearing in the first term is given by
  \begin{align*}
    K_{\lambda}(x,y)=\frac{\lambda^d}{(2\pi)^{d+1}}\int_{\R^d}\int_{-\infty}^{\infty}\widehat{\chi}(t)\e^{\I\lambda[\phi(x,y,\xi)+t p(y,\xi)/\lambda]}q(t/\lambda,x,y,\lambda\xi)\rd t\rd \xi. 
  \end{align*}
  Since $q\in S^0$, all $t$ and $\xi$ derivatives of $q(t/\lambda,x,y,\lambda\xi)$ are uniformly bounded in $\lambda$. By~\eqref{Taylor expansion of phase function},
  \begin{align*}
    |\partial_{\xi}[\phi(x,y,\xi)+t p(y,\xi)/\lambda]|\gtrsim |x-y|-\mathcal{O}(\lambda^{-1})
  \end{align*}
  on the support of the integrand.
  Therefore, if $|x-y|\geq C\lambda^{-1}$ with $C$ sufficiently large, nonstationary phase yields
  \begin{align*}
    |K_{\lambda}(x,y)|\lesssim_N 
    \lambda^d (1+\lambda|x-y|)^{-N}.
  \end{align*}
  The claim follows.
\end{proof}

Consider a cover $\{B_j\}$ of $M$ by geodesic balls of radius $\lambda^{-1}$, and let $\{\psi_j\}$ be a partition of unity subordinate to this cover. 
We define the weight functions
\begin{align}\label{def. w_j}
  w_{j}(x)=\lambda^d(1+\lambda\dist(x,B_j))^{-100d},\quad x\in M.
\end{align}

\begin{lemma}\label{lemma locally constant}
  Let $f\in L^2(M)$, and assume that $f$ is spectrally localised to frequencies at most $\lambda$, with respect to $P$. Then, for every $j$,
  \begin{align*}
    \|f\|_{L^{\infty}(B_j)}\lesssim \|f\|_{L^1(w_j)}+\|f\|_{L^1(M)}.
  \end{align*}
\end{lemma} 

\begin{proof}
This is an immediate consequence of Lemma~\ref{verificationlocalconstancyspectralcluster1}.
\end{proof}

\begin{lemma}\label{verificationlocalconstancyspectralcluster2}
  Let $f\in L^2(M)$, and assume that $f$ is spectrally localised to frequencies at most $\lambda\geq1$, with respect to $P$. Let $\Lambda\subset M$ be a set of $\lambda^{-1}$-separated points. Then for any $p\in [1,\infty]$,
  \begin{align*}
    \|f\|_{\ell^p(\Lambda)}\lesssim \lambda^{d/p}\|f\|_{L^p(M)}.
  \end{align*}
\end{lemma}

\begin{proof}
  Using Lemma~\ref{lemma locally constant}, the estimate for the first term follows as in the proof of Proposition~\ref{prop. abstract comparison ellp Lp}. The estimate for the second term is trivial since $\#\Lambda\lesssim\lambda^d$ and $\|f\|_{L^1(M)}\lesssim \|f\|_{L^p(M)}$ by H\"older's inequality and $\lambda\geq1$.
\end{proof}


\def\cprime{$'$}

\end{document}